\documentclass[12pt]{article}
\usepackage{fullpage,amsmath,amsthm,amsfonts,amssymb,graphicx,amscd,float,hyperref,enumitem}

\newtheorem{thm}{Theorem}[section]
\newtheorem{lem}[thm]{Lemma}
\newtheorem{qst}[thm]{Question}
\newtheorem{prop}[thm]{Proposition}
\newtheorem{cor}[thm]{Corollary}
\theoremstyle{definition}
\newtheorem{df}[thm]{Definition}
\newtheorem{rk}[thm]{Remark}
\newtheorem{ex}[thm]{Example}
\newtheorem{nt}[thm]{Standard Notation/Terminology}
\newtheorem*{mainthm}{Theorem}
\newtheorem*{mainprop}{Proposition}

\begin{document}

\title{Ideal Whitehead Graphs in $Out(F_r)$ I: Some Unachieved Graphs}
\author{Catherine Pfaff}
\date{}
\maketitle

\abstract{In \cite{ms93}, Masur and Smillie proved precisely which singularity index lists arise from pseudo-Anosov mapping classes.  In search of an analogous theorem for outer automorphisms of free groups, Handel and Mosher ask in \cite{hm11}: Is each connected, simplicial, ($2r-1$)-vertex graph the ideal Whitehead graph of a fully irreducible $\phi \in Out(F_r)$?  We answer this question in the negative by exhibiting, for each $r$, examples of connected (2r-1)-vertex graphs that are not the ideal Whitehead graph of any fully irreducible $\phi \in Out(F_r)$.  In the course of our proof we also develop machinery used in \cite{p12a} to fully answer the question in the rank-three case.}

\section{Introduction}

For a compact surface $S$, the \emph{mapping class group $\mathcal{MCG}(S)$} is the group of isotopy classes of homeomorphisms $h \colon S\to S$. A generic (see, for example, \cite{m11}) mapping class is \emph{pseudo-Anosov}, i.e. has a representative leaving invariant a pair of transverse measured singular minimal foliations. From the foliation comes a singularity index list. Masur and Smillie determined precisely which singularity index lists, permitted by the Poincare-Hopf index formula, arise from pseudo-Anosovs \cite{ms93}. The search for an analogous theorem in the setting of an outer automorphism group of a free group is still open.

We let $Out(F_r)$ denote the outer automorphism group of the free group of rank r. Analogous to pseudo-Anosov mapping classes are fully irreducible outer automorphisms, i.e. those such that no power leaves invariant the conjugacy class of a proper free factor. In fact, some fully irreducible outer automorphisms, called \emph{geometrics}, are induced by pseudo-Anosovs. The index lists of geometrics are understood through the Masur-Smillie index theorem.

In \cite{gjll}, Gaboriau, Jaeger, Levitt, and Lustig defined singularity indices for fully irreducible outer automorphisms. Additionally, they proved an $Out(F_r)$-analogue to the Poincare-Hopf index equality, namely the index sum inequality $i(\phi) \geq 1-r$ for a fully irreducible $\phi \in Out(F_r)$.

Having an inequality, instead of just an equality, makes the search for an analogue to the Masur-Smillie theorem richly more complicated. Toward this goal, Handel and Mosher asked in \cite{hm11}:

\begin{qst}{\label{Q:Q1}}
Which index types, satisfying $i(\phi) > 1-r$, are achieved by nongeometric fully irreducible $\phi \in Out(F_r)$?
\end{qst}

There are several results on related questions. For example, \cite{jl09} gives examples of automorphisms with the maximal number of fixed points on $\partial F_r$, as dictated by a related inequality in \cite{gjll}. However, our work focuses on an $Out(F_r)$-version of the Masur-Smillie theorem. Hence, in this paper, in \cite{p12c}, and in \cite{p12d} we restrict attention to fully irreducibles and the \cite{gjll} index inequality.

Beyond the existence of an inequality, instead of just an equality, ``ideal Whitehead graphs'' give yet another layer of complexity for fully irreducibles. An ideal Whitehead graph describes the structure of singular leaves, in analogue to the boundary curves of principle regions in Nielsen theory \cite{n86}. In the surface case, ideal Whitehead graphs are all circles. However, the ideal Whitehead graph $\mathcal{IW}(\phi)$ for a fully irreducible $\phi \in Out(F_r)$ (see \cite{hm11} or Definition \ref{D:whiteheadgraphs} below) gives a strictly finer outer automorphism invariant than just the corresponding index list. Indeed, each connected component $C_i$ of $\mathcal{IW}(\phi)$ contributes the index $1-\frac{k_i}{2}$ to the list, where $C_i$ has $k_i$ vertices. One can see many complicated ideal Whitehead graph examples, including complete graphs in every rank (in \cite{p12c}) and in the eighteen of the twenty-one connected, five-vertex graphs achieved by fully irreducibles in rank-three (\cite{p12d}). The deeper, more appropriate question is thus:

\begin{qst}{\label{Q:Q2}}
Which isomorphism types of graphs occur as the ideal Whitehead graph $\mathcal{IW}(\phi)$ of a fully irreducible outer automorphism $\phi$?
\end{qst}

\cite{p12d} will give a complete answer to Question \ref{Q:Q2} in rank 3 for the single-element index list $(-\frac{3}{2})$. In Theorem \ref{T:MainTheorem} of this paper we provide examples in each rank of connected (2r-1)-vertex graphs that are not the ideal Whitehead graph $\mathcal{IW}(\phi)$ for any fully irreducible $\phi \in Out(F_r)$, i.e. that are \emph{unachieved} in rank r:

\begin{mainthm}  \emph{For each $r \geq 3$, let $\mathcal{G}_r$ be the graph consisting of $2r-2$ edges adjoined at a single vertex.}
~\\
\vspace{-\baselineskip}
\begin{description}
\item [A.] \emph{For no fully irreducible $\phi \in Out(F_r)$ is $\mathcal{IW}(\phi) \cong \mathcal{G}_r$.} \\[-5.5mm]
\item [B.] \emph{The following connected graphs are not the ideal Whitehead graph $\mathcal{IW}(\phi)$ for any fully irreducible $\phi \in Out(F_3)$:} \\[-6mm]

\centering
\includegraphics[width=2.6in]{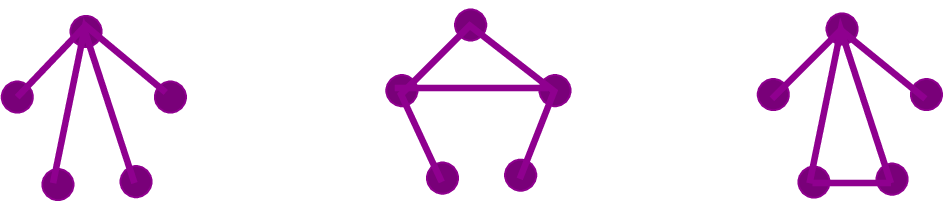}
\end{description}
\end{mainthm}

Nongeometric fully irreducible outer automorphisms are either ``ageometric'' or ``parageometric,'' as defined by Lustig. Ageometric outer automorphisms are our focus, since the index sum for a parageometric, as is true for a geometric, satisfies the Poincare-Hopf equality \cite{gjll}. Parageometrics have been studied in papers including \cite{hm07}. In \cite{bf94}, Bestvina and Feighn prove the \cite{gjll} index inequality is strict for ageometrics.

For a fully irreducible $\phi \in Out(F_r)$, to have the index list $(\frac{3}{2}-r)$, $\phi$ must be ageometric with a connected, (2r-1)-vertex ideal Whitehead graph $\mathcal{IW}(\phi)$. We chose to focus on the single-element index list $(\frac{3}{2}-r)$ because it is the closest to that achieved by geometrics, without being achieved by a geometric. We denote the set of connected (2r-1)-vertex, simplicial graphs by $\mathcal{PI}_{(r; (\frac{3}{2}-r))}$.

One often studies outer automorphisms via geometric representatives. Let $R_r$ be the $r$-petaled rose, with its fundamental group identified with $F_r$. For a finite graph $\Gamma$ with no valence-one vertices, a homotopy equivalence $R_r \to \Gamma$ is called a \emph{marking}. Such a graph $\Gamma$, together with its marking $R_r \to \Gamma$, is called a \emph{marked graph}. Each $\phi \in Out(F_r)$ can be represented by a homotopy equivalence $g\colon \Gamma \to \Gamma$ of a marked graph ($\phi= g_{*}\colon \pi_1(\Gamma) \to \pi_1(\Gamma)$). Thurston defined such a homotopy equivalence to be a \emph{train track map} when $g^k$ is locally injective on edge interiors for each $k>0$. When $g$ induces $\phi \in Out(F_r)$ and sends vertices to vertices, one says $g$ is a \emph{train track (tt) representative} for $\phi$ \cite{bh92}.

To prove Theorem \ref{T:MainTheorem}A, we give a necessary \emph{Birecurrency Condition} (Proposition \ref{P:BC}) on ``lamination train track structures.'' For a train track representative $g \colon \Gamma \to \Gamma$ on a marked rose, we define a \emph{lamination train track (ltt) Structure} \emph{$G(g)$} obtainable from $\Gamma$ by replacing the vertex $v$ with the ``local Whitehead graph'' $\mathcal{LW}(g; v)$. The local Whitehead graph encodes how lamination leaves enter and exit $v$. In our circumstance, $\mathcal{IW}(\phi)$ will be a subgraph of $\mathcal{LW}(g; v)$, hence of $G(g)$.

The lamination train track structure $G(g)$ is given a smooth structure so that leaves of the expanding lamination are realized as locally smoothly embedded lines. It is called \emph{birecurrent} if it has a locally smoothly embedded line crossing each edge infinitely many times as $\bold{R}\to \infty$ and as $\bold{R}\to -\infty$.

\begin{mainprop} \textbf{(Birecurrency Condition)} \emph{The lamination train track structure for each train track representative of each fully irreducible outer automorphism $\phi \in Out(F_r)$ is birecurrent.} \end{mainprop}

Combinatorial proofs (not included here) of Theorem \ref{T:MainTheorem}A exist. However, we include a proof using the Birecurrency Condition to highlight what we have observed to be a significant obstacle to achievability, namely the birecurrency of ltt structures. The Birecurrency Condition is also used in our proof of Theorem \ref{T:MainTheorem}B. We use it in \cite{p12c}, where we prove the achievability of the complete graph in each rank. Finally, the condition is used in \cite{p12d} to prove precisely which of the twenty-one connected, simplicial, five-vertex graphs are $\mathcal{IW}(\phi)$ for fully irreducible $\phi \in Out(F_3)$.

In Proposition \ref{P:IdealDecomposition} we show that each $\phi$, such that $\mathcal{IW}(\phi) \in \mathcal{PI}_{(r;(\frac{3}{2}-r))}$, has a power $\phi^R$ with a rotationless representative whose Stallings fold decomposition (see Subsection \ref{SS:StallingsFoldDecompositions}) consists entirely of proper full folds of roses (see Subsection \ref{SS:Folds}). The representatives of Proposition \ref{P:IdealDecomposition} are called ``ideally decomposable.'' We define in Section \ref{Ch:AMDiagrams} automata, ideal ``decomposition ($\mathcal{ID}$) diagrams'' with ltt structures as nodes. Every ideally decomposed representative is realized by a loop in an $\mathcal{ID}$ diagram. To prove Theorem \ref{T:MainTheorem}B we show ideally decomposed representatives cannot exist by showing that the $\mathcal{ID}$ diagrams do not have the correct kind of loops.

We again use the ideally decomposed representatives and $\mathcal{ID}$ diagrams in \cite{p12c} and \cite{p12d} to construct ideally decomposed representatives with particular ideal Whitehead graphs.

To determine the edges of the $\mathcal{ID}$ diagrams, we prove in Section \ref{Ch:AMProperties} a list of ``Admissible Map (AM) properties'' held by ideal decompositions. In Section \ref{Ch:Peels} we use the AM properties to determine the two geometric moves one applies to ltt structures in defining edges of the $\mathcal{ID}$ diagrams. The geometric moves turn out to have useful properties expanded upon in \cite{p12c} and \cite{p12d}.

\subsection*{Acknowledgements}

\indent The author would like to thank Lee Mosher for his truly invaluable conversations and Martin Lustig for his interest in her work. She also extends her gratitude to Bard College at Simon's Rock and the CRM for their hospitality.

\section{Preliminary definitions and notation}{\label{Ch:PrelimDfns}}

We continue with the introduction's notation. \emph{\textbf{Further we assume throughout this document that all representatives $g$ of $\phi \in Out(F_r)$ are train tracks (tts).}}

We let $\mathcal{FI}_r$ denoted the subset of $Out(F_r)$ consisting of all fully irreducible elements.

\vskip10pt

\noindent \textbf{2.1. Directions and turns}

\vskip1pt

In general we use the definitions from \cite{bh92} and \cite{bfh00} when discussing train tracks. We give further definitions and notation here. $g: \Gamma \to \Gamma$ will represent $\phi \in Out(F_r)$.

$\mathcal{E}^+(\Gamma)= \{E_1, \dots, E_{n}\}= \{e_1, e_1, \dots, e_{2n-1}, e_{2n} \}$ will be the edge set of $\Gamma$ with some prescribed orientation. For $E \in \mathcal{E}^+(\Gamma)$, $\overline{E}$ will be $E$ oppositely oriented. $\mathcal{E}(\Gamma)$:$=\{E_1, \overline{E_1}, \dots, E_n, \overline{E_n} \}$. If the indexing $\{E_1, \dots, E_{n}\}$ of the edges (thus the indexing $\{e_1, e_1, \dots, e_{2n-1}, e_{2n} \}$) is prescribed, we call $\Gamma$ an \emph{edge-indexed} graph. Edge-indexed graphs differing by an index-preserving homeomorphism will be considered equivalent.

$\mathcal{V}(\Gamma)$ will denote the vertex set of $\Gamma$ ($\mathcal{V}$, when $\Gamma$ is clear) and $\mathcal{D}(\Gamma)$ will denote $\underset{v \in \mathcal{V}(\Gamma)}{\cup} \mathcal{D}(v)$, where $\mathcal{D}(v)$ is the set of directions (germs of initial edge segments) at $v$.

For each $e \in \mathcal{E}(\Gamma)$, $D_0(e)$ will denote the initial direction of $e$ and $D_0 \gamma := D_0(e_1)$ for each path $\gamma=e_1 \dots e_k$ in $\Gamma$. \emph{$Dg$} will denote the direction map induced by $g$. We call $d \in \mathcal{D}(\Gamma)$ \emph{periodic} if $Dg^k(d)=d$ for some $k>0$ and \emph{fixed} if $k=1$.

$Per(x)$ will consist of the periodic directions at an $x \in \Gamma$ and $Fix(x)$ of those fixed. $Fix(g)$ will denote the fixed point set for $g$.

$\mathcal{T}(v)$ will denote the set of turns (unordered pairs of directions) at a $v \in \mathcal{V}(\Gamma)$ and $D^tg$ the induced map of turns. For a path $\gamma=e_1e_2 \dots e_{k-1}e_k$ in $\Gamma$, we say $\gamma$ \emph{contains (or crosses over)} the turn $\{\overline{e_i}, e_{i+1}\}$ for each $1 \leq i < k$. Sometimes we abusively write $\{\overline{e_i}, e_j\}$ for $\{D_0(\overline{e_i}), D_0(e_j)\}$. Recall that a turn is called \emph{illegal} for $g$ if $Dg^k(d_i)=Dg^k(d_j)$ for some $k$ ($d_i$ and $d_j$ are in the same \emph{gate}).

\vskip10pt

\noindent \textbf{2.2. Periodic Nielsen paths and ageometric outer automorphisms}

\vskip1pt

Recall \cite{bf94} that a \emph{periodic Nielsen path (pNp)} is a nontrivial path $\rho$ between $x,y \in Fix(g)$ such that, for some $k$, $g^k(\rho) \simeq \rho$ rel endpoints (\emph{Nielsen path (Np)} if $k=1$). In later sections we use \cite{gjll} that a $\phi \in \mathcal{FI}_r$ is ageometric if and only if some $\phi^k$ has a representative with no pNps (closed or otherwise). $\mathcal{AFI}_r$ will denote the subset of $\mathcal{FI}_r$ consisting precisely of its ageometric elements.

\vskip10pt

\noindent \textbf{2.3. Local Whitehead graphs, local stable Whitehead graphs, and ideal Whitehead graphs}{\label{S:IWGs}}

\vskip1pt

Please note that the ideal Whitehead graphs, local Whitehead graphs, and stable Whitehead graphs used here (defined in \cite{hm11}) differ from other Whitehead graphs in the literature. We clarify a difference. In general, Whitehead graphs record turns taken by immersions of 1-manifolds into graphs. In our case, the 1-manifold is a set of lines, the attracting lamination. In much of the literature the 1-manifolds are circuits representing conjugacy classes of free group elements. For example, for the Whitehead graphs of \cite{cv86}, edge images are viewed as cyclic words. This is not true for ours.

The following can be found in \cite{hm11}, though it is not their original source, and versions here are specialized. See \cite{p12a} for more extensive explanations of the definitions and their invariance. \emph{\textbf{For this subsection $g: \Gamma \to \Gamma$ will be a pNp-free train track.}}

\begin{df}{\label{D:whiteheadgraphs}} Let $\Gamma$ be a connected marked graph, $v \in \Gamma$, and $g: \Gamma \to \Gamma$ a representative of $\phi \in Out(F_r)$. The \emph{local Whitehead graph} for $g$ at $v$ (denoted \emph{$\mathcal{LW}(g; v)$}) has:

(1) a vertex for each direction $d \in \mathcal{D}(v)$ and

(2) edges connecting vertices for $d_1, d_2 \in \mathcal{D}(v)$ where $\{d_1, d_2 \}$ is taken by some $g^k(e)$, with $e \in \mathcal{E}(\Gamma)$.

\noindent The \emph{local Stable Whitehead graph} $\mathcal{SW}(g; v)$ is the subgraph obtained by restricting precisely to vertices with labels in $Per(v)$. For a rose $\Gamma$ with vertex $v$, we denote the single local stable Whitehead graph $\mathcal{SW}(g; v)$ by $\mathcal{SW}(g)$ and the single local Whitehead graph $\mathcal{LW}(g; v)$ by $\mathcal{LW}(g)$.

For a pNp-free $g$, the \emph{ideal Whitehead graph of $\phi$}, \emph{$\mathcal{IW}(\phi)$}, is isomorphic to $\underset{\text{singularities v} \in \Gamma}{\bigsqcup} \mathcal{SW}(g;v)$, where a \emph{singularity} for $g$ in $\Gamma$ is a vertex with at least three periodic directions. In particular, when $\Gamma$ is a rose, $\mathcal{IW}(\phi) \cong \mathcal{SW}(g)$.
\end{df}

\begin{ex}{\label{Ex:whiteheadgraphs}}
Let $g: \Gamma \to \Gamma$, where $\Gamma$ is a rose and $g$ is the train track such that the following describes the edge-path images of its edges:
$$g =
\begin{cases}
a \mapsto abacbaba\bar{c}abacbaba \\
b \mapsto ba\bar{c} \\
c \mapsto c\bar{a}\bar{b}\bar{a}\bar{b}\bar{a}\bar{b}\bar{c}\bar{a}\bar{b}\bar{a}c
\end{cases}.
$$

The vertices for $\mathcal{LW}(g)$ are $\{a, \bar a, b, \bar b, c, \bar c \}$ and the vertices of $\mathcal{SW}(g)$ are $\{a, \bar a, b, c, \bar c \}$: The periodic (actually fixed) directions for $g$ are $\{a, \bar a, b, c, \bar c \}$. $\bar b$ is not periodic since $Dg(\bar b)=c$, which is a fixed direction, meaning that $Dg^k(\bar b)=c$ for all $k \geq 1$, and thus $Dg^k(\bar{b})$ does NOT equal $\bar{b}$ for any $k \geq 1$.

The turns taken by the $g^k(E)$, for $E \in \mathcal{E}(\Gamma)$, are $\{a,\bar{b}\}$, $\{\bar{a},\bar{c}\}$, $\{b,\bar{a}\}$, $\{b,\bar{c}\}$, $\{c,\bar{a}\}$, and $\{a, c\}$. Since $\{a,\bar{b}\}$ contains the nonperiodic direction $\bar{b}$, this turn does not give an edge in $\mathcal{SW}(g)$, though does give an edge in $\mathcal{LW}(g)$. All other turns listed give edges in both $\mathcal{SW}(g)$ and $\mathcal{LW}(g)$.

$\mathcal{LW}(g)$ and $\mathcal{SW}(g)$ respectively look like (reasons for colors become clear in Subsection 2.4):
~\\
\vspace{-6.5mm}
\begin{figure}[H]
\centering
\noindent \includegraphics[width=2in]{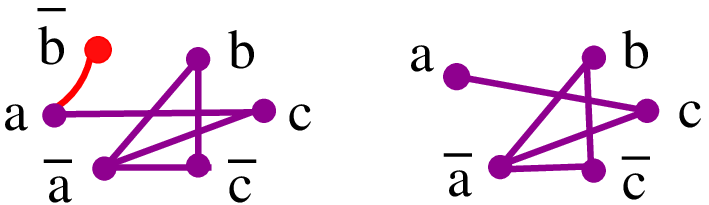}
\label{fig:IdealWhiteheadGraphs} \\[-2mm]
\end{figure}
\end{ex}

\vskip7pt

\noindent \textbf{2.4. Lamination train track structures}{\label{SS:Realltts}}

\vskip1pt

We define here ``lamination train track (ltt) structures.'' Bestvina, Feighn, and Handel discussed in their papers slightly different train track structures. However, those we define contain as smooth paths lamination (see \cite{bfh00}) leaf realizations. This makes them useful for deeming unachieved particular ideal Whitehead graphs and for constructing representatives (see \cite{p12c} and \cite{p12d}). \emph{\textbf{Again, $g: \Gamma \to \Gamma$ will be a pNp-free train track on a marked rose with vertex $v$.}}

The \emph{colored local Whitehead graph $\mathcal{CW}(g)$ at $v$}, is $\mathcal{LW}(g)$, but with the subgraph $\mathcal{SW}(g)$ colored purple and $\mathcal{LW}(g)- \mathcal{SW}(g)$ colored red (nonperiodic direction vertices are red).

Let $\Gamma_N=\Gamma-N(v)$ where $N(v)$ is a contractible neighborhood of $v$. For each $E_i \in \mathcal{E}^+$, add vertices $d_i$ and $\overline{d_i}$ at the corresponding boundary points of the partial edge $E_i-(N(v) \cap E_i)$. A \emph{lamination train track (ltt) Structure} \emph{$G(g)$} for $g$ is formed from $\Gamma_N \bigsqcup \mathcal{CW}(g)$ by identifying the vertex $d_i$ in $\Gamma_N$ with the vertex $d_i$ in $\mathcal{CW}(g)$. Vertices for nonperiodic directions are red, edges of $\Gamma_N$ black, and all periodic vertices purple.

An ltt structure $G(g)$ is given a \emph{smooth structure} via a partition of the edges at each vertex into two sets: $\mathcal{E}_b$ (containing the black edges of $G(g)$) and $\mathcal{E}_c$ (containing the colored edges of $G(g)$). A \emph{smooth path} we will mean a path alternating between colored and black edges.

An edge connecting a vertex pair $\{d_i, d_j \}$ will be denoted [$d_i, d_j$], with interior ($d_i, d_j$). Additionally, $[e_i]$ will denote the black edge [$d_i, \overline{d_i}$] for $e_i \in \mathcal{E}(\Gamma)$.

For a smooth (possibly infinite) path $\gamma$ in $G(g)$, the \emph{path (or line) in $\Gamma$ corresponding to $\gamma$} is $\dots e_{-j}e_{-j+1} \dots e_{-1}e_0e_1 \dots e_j \dots$, with $\gamma= \dots [d_{-j}, \overline{d_{-j}}][\overline{d_{-j}}, d_{-j+1}] \dots [d_0, \overline{d_0}][\overline{d_0}, d_1] \dots [d_j, \overline{d_j}] \dots,$ where each $d_i=D_0(e_i)$, each $[d_i, \overline{d_i}]$ is the black edge $[e_i]$, and each $[d_i, \overline{d_{i+1}}]$ is a colored edge. We denote such a path $\gamma=  [\dots, d_{-j}, \overline{d_{-j}}, d_{-j+1}, \dots, \overline{d_{-1}}, d_0, \overline{d_0}, d_1, \dots, d_j, \overline{d_j} \dots].$

\begin{ex}{\label{Ex:G(g)}} Let $g$ be as in Example \ref{Ex:whiteheadgraphs}. The vertex $\bar b$ in $G(g)$ is red. All others are purple. $G(g)$ has a purple edge for each edge in $\mathcal{SW}(g)$ and a single red edge for the turn $\{a,\bar{b}\}$ (represented by an edge in $\mathcal{LW}(g)$, but not in $\mathcal{SW}(g)$). $\mathcal{CW}(g)$ is $\mathcal{LW}(g)$ with the coloring of Example \ref{Ex:whiteheadgraphs}. And $G(g)$ is obtained from $\mathcal{CW}(g)$ by adding black edges connecting the vertex pairs $\{a,\bar{a}\}$, $\{b,\bar{b}\}$, and $\{c,\bar{c}\}$ (corresponding precisely to the edges $a, b,$ and $c$ of $\Gamma$).
~\\
\vspace{-7.2mm}
\begin{figure}[H]
\centering
\noindent \includegraphics[width=1in]{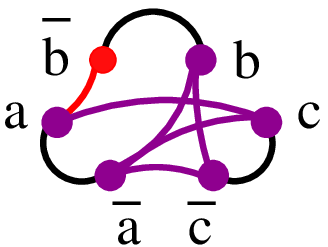}
\label{fig:lttExample} \\[-3mm]
\end{figure}
\indent Once can check that each $g(e)$ is realized by a smooth path in $G(g)$.
\end{ex}

\begin{rk} If $\Gamma$ had more than one vertex, one could define $G(g)$ by creating a colored graph $\mathcal{CW}(g;v)$ for each vertex, removing an open neighborhood of each vertex when forming $\Gamma_N$, and then continuing with the identifications as above in $\Gamma_N \bigsqcup (\cup \mathcal{CW}(g;v))$.
\end{rk}

\section{Ideal decompositions}{\label{Ch:IdealDecompositions}}

In this section we prove (Proposition \ref{P:IdealDecomposition}): if $\mathcal{G} \in \mathcal{PI}_{(r;(\frac{3}{2}-r))}$ is $\mathcal{IW}(\phi)$ for a $\phi \in \mathcal{AFI}_r$, then $\phi$ has a rotationless power with a representative satisfying several nice properties, including that its Stallings fold decomposition consists entirely of proper full folds of roses. We call such a decomposition an \emph{ideal decomposition}. Proving an ideal decomposition cannot exist will suffice to deem a $\mathcal{G}$ unachieved.

We remind the reader of definitions of folds and a Stallings fold decomposition before introducing ideal decompositions, as our Proposition \ref{P:IdealDecomposition} proof relies heavily upon them.

\subsection{Folds}{\label{SS:Folds}}

Stallings introduced ``folds'' in \cite{s83} and Bestvina and Handel use several versions in their train track algorithm of \cite{bh92}.

Let $g: \Gamma \to \Gamma$ be a homotopy equivalence of marked graphs. Suppose $g(e_1)=g(e_2)$ as edge paths, where $e_1, e_2 \in \mathcal{E}(\Gamma)$ emanate from a common vertex $v \in \mathcal{V} (\Gamma)$. One can obtain a graph $\Gamma_1$ by identifying $e_1$ and $e_2$ in such a way that $g:\Gamma \to \Gamma$ projects to $g_1: \Gamma_1 \to \Gamma_1$ under the quotient map induced by the identification of $e_1$ and $e_2$. $g_1$ is also a homotopy equivalence and one says $g_1$ and $\Gamma_1$ are obtained from $g$ by an \emph{elementary fold} of $e_1$ and $e_2$.

To generalize one requires $e_1' \subset e_1$ and $e_2' \subset e_2$ only be maximal, initial, nontrivial subsegments of edges emanating from a common vertex such that $g(e_1')=g(e_2')$ as edge paths and such that the terminal endpoints of $e_1$ and $e_2$ are in $g^{-1}(\mathcal{V}(\Gamma))$. Possibly redefining $\Gamma$ to have vertices at the endpoints of $e_1'$ and $e_2'$, one can fold $e_1'$ and $e_2'$ as $e_1$ and $e_2$ were folded above. We say $g_1\colon\Gamma_1 \to \Gamma_1$ is obtained by
~\\
\vspace{-6mm}
\begin{itemize}
\item a \emph{partial fold} of $e_1$ and $e_2$: if both $e_1'$ and $e_2'$ are proper subedges; \\[-6mm]
\item  a \emph{proper full fold} of $e_1$ and $e_2$: if only one of $e_1'$ and $e_2'$ is a proper subedge (the other a full edge); \\[-6mm]
\item an \emph{improper full fold} of $e_1$ and $e_2$: if $e_1'$ and $e_2'$ are both full edges.
\end{itemize}

\subsection{Stallings fold decompositions}{\label{SS:StallingsFoldDecompositions}}

Stallings \cite{s83} also showed a tight homotopy equivalence of graphs is a composition of elementary folds and a final homeomorphism. We call such a decomposition a \emph{Stallings fold decomposition}.

A description of a Stallings Fold Decomposition can be found in \cite{s89}, where Skora described a Stallings fold decomposition for a $g\colon \Gamma \to \Gamma'$ as a sequence of folds performed continuously. Consider a lift $\tilde{g}\colon \tilde{\Gamma} \to \tilde{\Gamma}'$, where here $\tilde{\Gamma}'$ is given the path metric. Foliate $\tilde{\Gamma}$ x $\tilde{\Gamma}'$ with the leaves $\tilde{\Gamma}$ x $\{x'\}$ for $x' \in \Gamma'$. Define $N_t(\tilde{g})=\{(x,x') \in \tilde{\Gamma}$ x $\tilde{\Gamma}'$ $\vert$ $d(\tilde{g}(x),x') \leq t \}$. For each $t$, by restricting the foliation to $N_t$ and collapsing all leaf components, one obtains a tree $\Gamma_t$. Quotienting by the $F_r$-action, one sees the sequence of folds performed on the graphs below over time.

Alternatively, at an illegal turn for $g\colon \Gamma  \to \Gamma$, fold maximal initial segments having the same image in $\tilde{\Gamma}'$ to obtain a map $g^1: \Gamma_1 \to \Gamma'$ of the quotient graph $\Gamma_1$. Repeat for $g^1$. If some $g^k$ has no illegal turn, it will be a homeomorphism and the fold sequence is complete. Using this description, we can assume only the final element of the decomposition is a homeomorphism. Thus, a Stallings fold decomposition of $g:\Gamma \to \Gamma$ can be written $\Gamma_0 \xrightarrow{g_1} \Gamma_1 \xrightarrow{g_2} \cdots \xrightarrow{g_{n-1}} \Gamma_{n-1} \xrightarrow{g_n} \Gamma_n$ where each $g_k$, with $1 \leq k \leq n-1$, is a fold and $g_n$ is a homeomorphism.

\subsection{Ideal Decompositions}{\label{SS:Folds}}

In this subsection we prove Proposition \ref{P:IdealDecomposition}. For the proof, we need \cite{hm11}: For $\phi \in \mathcal{AFI}_r$ such that $\mathcal{IW}(\phi) \in \mathcal{PI}_{(r;(\frac{3}{2}-r))}$, $\phi$ is \emph{rotationless} if and only if the vertices of $\mathcal{IW}(\phi)$ are fixed by the action of $\phi$. We also need that a representative $g$ of $\phi \in Out(F_r)$ is rotationless if and only if $\phi$ is rotationless. Finally, we need the following lemmas.

\begin{lem}{\label{L:pNpFreePreserved}} Let $g \colon \Gamma \to \Gamma$ be a pNp-free tt representative of $\phi \in \mathcal{FI}_r$ and $\Gamma = \Gamma_0 \xrightarrow{g_1} \Gamma_1 \xrightarrow{g_2} \cdots \xrightarrow{g_{n-1}} \Gamma_{n-1} \xrightarrow{g_n} \Gamma_n = \Gamma$ a decomposition of $g$ into homotopy equivalences of marked graphs with no valence-one vertices. Then the composition $h \colon \Gamma_k \xrightarrow{g_{k+1}} \Gamma_{k+1} \xrightarrow{g_{k+2}} \cdots \xrightarrow{g_{k-1}} \Gamma_{k-1} \xrightarrow{g_k} \Gamma_k$ is also a pNp-free tt representative of $\phi$ (in particular, $\mathcal{IW}(h) \cong \mathcal{IW}(g)$).
\end{lem}

\begin{proof} Suppose $h$ had a pNp $\rho$ and $h^p(\rho) \simeq \rho$ rel endpoints. Let $\rho_1=g_n \circ \cdots \circ g_{k+1}(\rho)$.  If $\rho_1$ were trivial, $h^p(\rho)=(g_k \circ \cdots \circ g_1 \circ g^{p-1})(g_n \circ \cdots \circ g_{k+1}(\rho))=(g_k \circ \cdots \circ g_1 \circ g^{p-1})(\rho_1)$ would be trivial, contradicting $\rho$ being a pNp. So assume $\rho_1$ is not trivial.

$g^p(\rho_1)=g^p((g_k \circ \cdots \circ g_1)(\rho))=(g_n \circ \cdots \circ g_{k+1}) \circ h^p(\rho)$.  Now, $h^p(\rho) \simeq \rho$ rel endpoints and so $(g_n \circ \cdots \circ g_{k+1}) \circ h^p(\rho) \simeq (g_n \circ \cdots \circ g_{k+1})(\rho)$ rel endpoints. So $g^p(\rho_1)=g^p((g_k \circ \cdots \circ g_1)(\rho))=(g_n \circ \cdots \circ g_{k+1}) \circ h^p(\rho)$ is homotopic to $(g_n \circ \cdots \circ g_{k+1})(\rho)=\rho_1$ rel endpoints. This makes $\rho_1$ a pNp for $g$, contradicting that $g$ is pNp-free. Thus, $h$ is pNp-free.

Let $\pi\colon R_r \to \Gamma$ mark $\Gamma_1$. Since $g_1$ is a homotopy equivalence, $g_1 \circ \pi$ gives a marking on $\Gamma$. So $g$ and $h$ differ by a change of marking and thus represent the same outer automorphism $\phi$.

Finally, we show $h$ is a train track. For contradiction's sake suppose $h(e)$ crossed an illegal turn $\{d_1, d_2 \}$. Since each $g_j$ is necessarily surjective, some $(g_k \circ \cdots \circ g_1)(e_i)$ would traverse $e$. So $(g_k \circ \cdots \circ g_1)(e_i)$ would cross $\{d_1, d_2 \}$. And $g^2(e_i)=(g_n \circ \cdots \circ g_{k+1}) \circ h \circ (g_k \circ \cdots \circ g_1)(e_i)$ would cross $\{D(g_n \circ \cdots \circ g_{k+1})(d_1), D(g_n \circ \cdots \circ g_{k+1})(d_2) \}$, which would either be illegal or degenerate (since $\{d_1, d_2 \}$ is an illegal turn). This would contradict that $g$ is a tt.  So $h$ is a tt.  \qedhere
\end{proof}

\begin{lem}{\label{L:GateCollapsing}} Let $g: \Gamma \to \Gamma$ be a pNp-free tt representative of $\phi \in \mathcal{FI}_r$ with $2r-1$ fixed directions and Stallings fold decomposition $\Gamma_0 \xrightarrow{g_1} \Gamma_1 \xrightarrow{g_2} \cdots \xrightarrow{g_{n-1}} \Gamma_{n-1} \xrightarrow{g_n} \Gamma_n$. Let $g^i$ be such that $g=g^i \circ g_i \circ \cdots \circ g_1$. Let $d_{(1,1)}, \dots, d_{(1,2r-1)}$ be the fixed directions for $Dg$ and let $d_{j,k}=D(g_j \circ \cdots \circ g_1)(d_{1,k})$ for each $1 \leq j \leq n$ and $1 \leq k \leq 2r-1$. Then $D(g^i)$ is injective on $\{d_{(i,1)}, \dots, d_{(i,2r-1)}\}$.
\end{lem}

\begin{proof} Let $d_{(1,1)}, \dots, d_{(1,2r-1)}$ be the fixed directions for $Df$. If $D(g^i)$ identified any of $d_{(i,1)}, \dots, d_{(i,2r-1)}$, then $Df$ would have fewer than 2r-1 directions in its image. \qedhere
\end{proof}

\begin{prop}{\label{P:IdealDecomposition}} Let $\phi \in Out(F_r)$ be an ageometric, fully irreducible outer automorphism whose ideal Whitehead graph $\mathcal{IW}(\phi)$ is a connected, (2r-1)-vertex graph. Then there exists a train track representative of a power $\psi=\phi^R$ of $\phi$ that is:
~\\
\vspace{-\baselineskip}
\begin{enumerate}[itemsep=-1.5pt,parsep=3pt,topsep=3pt]
\item on the rose, \\[-5.5mm]
\item rotationless, \\[-5.5mm]
\item pNp-free, and \\[-5.5mm]
\item decomposable as a sequence of proper full folds of roses.
\end{enumerate}
\noindent In fact, it decomposes as $\Gamma = \Gamma_0 \xrightarrow{g_1} \Gamma_1 \xrightarrow{g_2} \cdots \xrightarrow{g_{n-1}}
\Gamma_{n-1} \xrightarrow{g_n} \Gamma_n = \Gamma$, where: \newline
\noindent (I) the index set $\{1, \dots, n \}$ is viewed as the set $\mathbf {Z}$/$n \mathbf {Z}$ with its natural cyclic ordering; \newline
\noindent (II) each $\Gamma_k$ is an edge-indexed rose with an indexing $\{e_{(k,1)}, e_{(k,2)}, \dots, e_{(k,2r-1)}, e_{(k,2r)}\}$ where:
~\\
\vspace{-\baselineskip}
{\begin{description}
\item (a) one can edge-index $\Gamma$ with $\mathcal{E}(\Gamma)=\{e_1, e_2, \dots, e_{2r-1}, e_{2r} \}$ such that, for each $t$ with $1 \leq t \leq 2r$, $g(e_t)=e_{i_1} \dots e_{i_s}$ where $(g_n \circ \cdots \circ g_1)(e_{0,t})=e_{n, i_1} \dots e_{n, i_s}$; \\[-5mm]
\item (b) for some $i_k,j_k$ with $e_{k,i_k} \neq (e_{k,j_k})^{\pm 1}$
$$g_k(e_{k-1,t}):=
\begin{cases}
e_{k,t} e_{k,j_k} \text{ for $t=i_k$} \\
e_{k,t} \text{ for all $e_{k-1,t} \neq e_{k-1,j_k}^{\pm 1}$; and}
\end{cases}$$
\noindent (the edge index permutation for the homeomorphism in the decomposition is trivial, so left out)
\item (c) for each $e_t \in \mathcal{E}(\Gamma)$ such that $t \neq j_n$, we have $Dh(d_t)=d_t$, where $d_t=D_0(e_t)$.
\end{description}}
\end{prop}

\begin{proof} Since $\phi \in \mathcal{AFI}_r$, there exists a pNp-free tt representative $g$ of a power of $\phi$. Let $h=g^k: \Gamma \to \Gamma$ be rotationless. Then $h$ is also a pNp-free tt representative of some $\phi^R$ and $h$ (and all powers of $h$) satisfy (2)-(3). Since $h$ has no pNps (meaning $\mathcal{IW}(\phi^R) \cong \underset{\text{singularities v} \in \Gamma}{\bigsqcup} \mathcal{SW}(h;v)$ and, if $\Gamma$ is the rose, $\mathcal{SW}(h) \cong \mathcal{IW}(\phi^R)$ ), since $h$ fixes all its periodic directions, and since $\mathcal{IW}(\phi)$ (hence $\mathcal{IW}(\phi^R)$) is in $\mathcal{PI}_{(r;(\frac{3}{2}-r))}$, $\Gamma$ must have a vertex with $2r-1$ fixed directions. Thus, $\Gamma$ must be one of:
~\\
\vspace{-7mm}
\begin{figure}[H]
\centering
\noindent \includegraphics[width=3.3in]{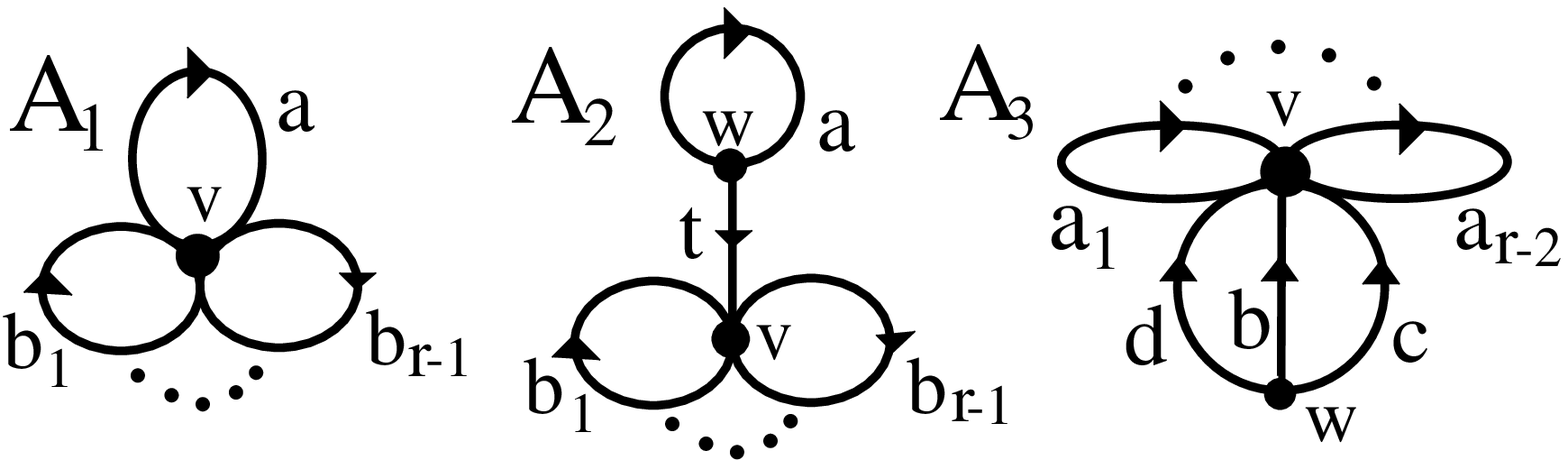}
\label{fig:HigherRankGraphChoices} \\[-3.75mm]
\end{figure}

If $\Gamma=A_1$, $h$ satisfies (3). We show, in this case, we also have the decomposition for (4). However, first we show $\Gamma$ cannot be $A_2$ or $A_3$ by ruling out all possibilities for folds in $h$'s Stallings decomposition.

If $\Gamma=A_2$, $v$ has to be the vertex with 2r-1 fixed directions. $h$ has an illegal turn unless it it is a homeomorphism, contradicting irreducibility. Note $w$ could not be mapped to $v$ in a way not forcing an illegal turn at $w$, as this would force either an illegal turn at $v$ (if $t$ were wrapped around some $b_i$) or we would have backtracking on $t$. Because all 2r-1 directions at $v$ are fixed by $h$, if $h$ had an illegal turn, it would have to occur at $w$ (no two fixed directions can share a gate).

The turns at $w$ are $\{a, \bar{a}\}$, $\{a, t\}$, and $\{\bar{a}, t\}$. By symmetry we only need to rule out illegal turns at $\{a, \bar{a}\}$ and $\{a, t\}$.

First, suppose $\{a, \bar{a}\}$ were illegal and the first fold in the Stallings decomposition. Fold $\{a, \bar{a}\}$ maximally to obtain $(A_2)_1$. Completely collapsing $a$ would change the homotopy type of $A_2$.
~\\
\vspace{-7mm}
\begin{figure}[H]
\centering
\noindent \includegraphics[width=3.5in]{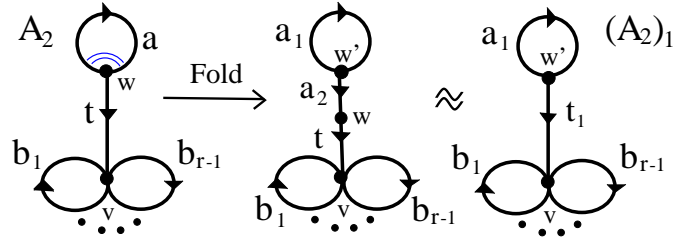}
\hspace{1.6in}\parbox{5.75in}{\caption{{\small{\emph{$a_1$ is the portion of $a$ not folded, $a_2$ is the edge created by the fold, $w'$ is the vertex created by the fold, and $t_1$ is $a_2 \cup t$ without the (now unnecessary) vertex $w$}}}}}
\label{fig:RosewithStemPetalFold} \\[-3mm]
\end{figure}

Let $h_1: (A_2)_1 \to (A_2)_1$ be the induced map of [BH92]. Since the fold of $\{a, \bar{a}\}$ was maximal, $\{a_1, \overline{a_1}\}$ must be legal. Since $h$ was a train track, $\{t_1, a_1\}$ and $\{t_1, \overline{a_1}\}$ would also be legal. But then $h_1$ would fix all directions at both vertices of $\Gamma_1$ (since it still would need to fix all directions at $v$). This would make $h_1$ a homeomorphism, again contradicting irreducibility. So $\{a, \bar{a}\}$ could not have been the first turn folded.  We are left to rule out $\{a, t\}$.

Suppose the first turn folded in the Stallings decomposition were $\{a, t\}$. Fold $\{a, t\}$ maximally to obtain $(A_2)'_1$. Let $h_1'\colon (A_2)'_1 \to (A_2)'_1$ be the induced map of [BH92]. Either \newline
\indent \indent \indent A. all of $t$ was folded with a full power of $a$; \newline
\indent \indent \indent B. all of $t$ was folded with a partial power of $a$; or \newline
\indent \indent \indent C. part of $t$ was folded with either a full or partial power of $a$.

If (A) or (B) held, $(A_2)_1'$ would be a rose and $h_1'$ would give a representative on the rose, returning us to the case of $A_1$. So we just need to analyze (C).

Consider first (C), i.e. suppose that part of $t$ is folded with either a full or partial power of $a$:
~\\
\vspace{-6.8mm}
\begin{figure}[H]
\centering
\noindent \includegraphics[width=3.8in]{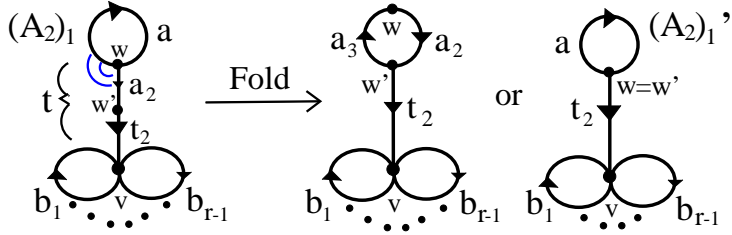}
\hspace{1.6in}\parbox{5.8in}{\caption{{\small{\emph{Of the two scenarios on the right, the leftmost is where the fold ends in the middle of $a$. $a_2$ is a possible portion of $a$ folded with the portion of $t$, $a_3$ would be the portion of $a$ not folded with $t$, and $t_2$ would be the portion of $t$ not folded with $a$}}}}}
\label{fig:RoseWithStemChoices2} \\[-3mm]
\end{figure}

If $h=h^1 \circ g_1$, where $g_1$ is the single fold performed thus far, then $h^1$ could not identify any directions at $w'$: identifying $a_2$ and $t_2$ would lead to $h$ back-tracking on $t$; identifying $t_2$ and $\bar{a}$ would lead to $h$ back-tracking on $a$; and $h^1$ could not identify $t_2$ and $\overline{a_3}$ because the fold was maximal. But then all directions of $(A_2)_{1}'$ would be fixed by $h^1$, making $h^1$ a homeomorphism and the decomposition complete. However, this would make $h$ consist of the single fold $g_1$ and a homeomorphism, contradicting $h$'s irreducibility. Thus, all cases where $\Gamma = A_2$ are either impossible or yield the representative on the rose for (1).

Now assume $\Gamma=A_3$. $v$ must have $2r-1$ fixed directions. As with $A_2$, since $h$ must fix all directions at $v$, if $h$ had an illegal turn (which it still has to) it would be at $w$. Without losing generality assume $\{b, d\}$ is an illegal turn and that the first Stallings fold maximally folds $\{b, d\}$. Folding all of $b$ and $d$ would change the homotopy type. So assume (again without generality loss) either:
~\\
\vspace{-6.5mm}
\begin{itemize}
\item all of $b$ is folded with part of $d$ or \\[-6mm]
\item only proper initial segments of $b$ and $d$ are folded with each other. \\[-6mm]
\end{itemize}
\noindent If all of $b$ is folded with part of $d$, we get a pNp-free tt on the rose. So suppose only proper initial segments of $b$ and $d$ are identified. Let $h_1\colon (A_3)_1 \to (A_3)_1$ be the \cite{bh92} induced map.
~\\
\vspace{-7.25mm}
\begin{figure}[H]
\centering
\noindent \includegraphics[width=4.8in]{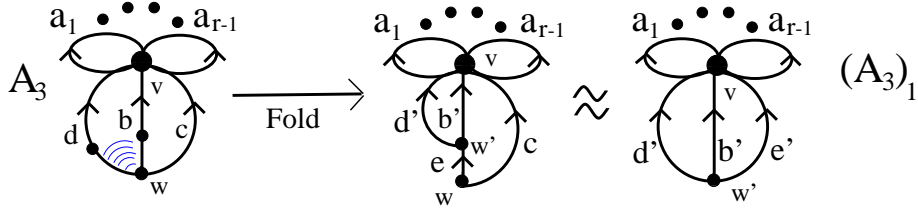}
~\\
\vspace{-2mm}
\parbox{6in}{\caption{{\small{\emph{$e$ was created by the fold and $e'$ is $\bar{e} \cup c$ without the (now unnecessary) vertex $w$}}}}}
\label{fig:TheFold} \\[-3mm]
\end{figure}

The new vertex $w'$ has 3 distinct gates: $\{b', d'\}$ is legal since the fold was maximal and $\{b', \bar{e}\}$ and $\{d', \bar{e}\}$ must be legal or $h$ would have back-tracked on $b$ or $d$, respectively. This leaves that the entire decomposition is a single fold and a homeomorphism, again contradicting $h$'s irreducibility.

We have ruled out $A_3$ and proved for (1) that we have a pNp-free representative on the rose of some $\psi=\phi^R$. We now prove (4).

Let $h$ be the pNp-free tt representative of $\phi^R$ on the rose and $\Gamma_0 \xrightarrow{g_1} \Gamma_1 \xrightarrow{g_2} \cdots \xrightarrow{g_{n-1}} \Gamma_{n-1} \xrightarrow{g_n} \Gamma_n$ the Stallings decomposition. Each $g_i$ is either an elementary fold or locally injective (thus a homeomorphism). We can assume $g_n$ is the only homeomorphism. Let $h^i=g_n \circ \dots \circ g_{i+1}$. Since $h$ has precisely $2r-1$ gates, $h$ has precisely one illegal turn.  We first determine what $g_1$ could be.  $g_1$ cannot be a homeomorphism or $h=g_1$, making $h$ reducible.  So $g_1$ must maximally fold the illegal turn. Suppose the fold is a proper full fold. (If it is not, see the analysis below of cases of improper or partial folds.)
~\\
\vspace{-7.2mm}
\begin{figure}[H]
\centering
\noindent \includegraphics[width=3.8in]{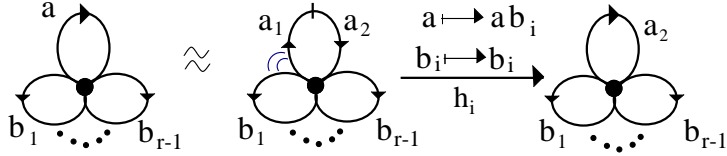}
~\\
\vspace{-2mm}
\parbox{6in}{\caption{\small{\emph{Proper full fold}}}
\label{fig:HigherRankRoseProperFullFold}} \\[-3mm]
\end{figure}

By Lemma \ref{L:GateCollapsing}, $h^1$ can only have one turn $\{d_1, d_2\}$ where $Dh^1(\{d_1, d_2\})$ is degenerate (we call such a turn an \emph{order-1 illegal turn} for $h^1$). If it has no order-1 illegal turn, $h^1$ is a homeomorphism and the decomposition is determined. So suppose $h^1$ has an order-1 illegal turn (with more than one, $h$ could not have 2r-1 distinct gates). The next Stallings fold must maximally fold this turn. With similar logic, we can continue as such until either $h$ is obtained, in which case the desired decomposition is found, or until the next fold is not a proper full fold. The next fold cannot be an improper full fold or the homotopy type would change. Suppose after the last proper full fold we have:
~\\
\vspace{-7.25mm}
\begin{figure}[H]
\centering
\noindent \includegraphics[width=1.2in]{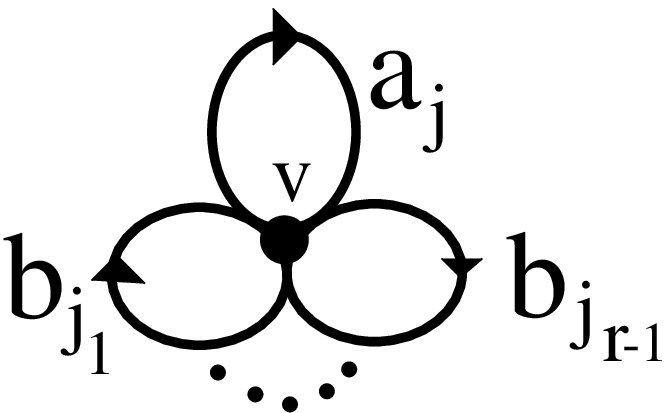}
\label{fig:AfterLastPFF} \\[-3mm]
\end{figure}

Without losing generality, suppose the illegal turn is $\{a_j, \overline{a_j}\}$.  Maximally folding $\{a_j, \overline{a_j}\}$ yields $A_2$, as above. This cannot be the final fold in the decomposition since $A_1$ is not homeomorphic to $A_2$. By Lemma \ref{L:pNpFreePreserved}, the illegal turn must be at $w$. The fold of Figure 3 cannot be performed, as our fold was maximal. If the fold of Figure \ref{fig:HigherRankRoseProperFullFold} were performed, there would be backtracking on $a$.

Now suppose, without loss of generality, that the first Stallings fold that is not a proper full fold is a partial fold of $b'$ and $c'$, as in the following figure.
~\\
\vspace{-7.1mm}
\begin{figure}[H]
\centering
\noindent \includegraphics[width=3in]{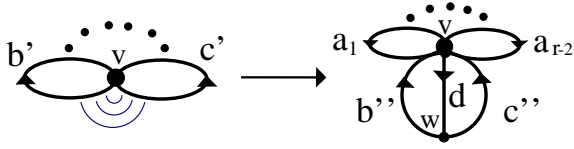}
~\\
\vspace{-2mm}
\parbox{5.6in}{\caption{{\small{\emph{$d$ is the edge created by folding initial segments of $b'$ and $c'$, $b''$ is the terminal segment of $b'$ not folded, and $c''$ is the terminal segment of $c'$ not folded}}}}}
\label{fig:ImproperFold} \\[-2.85mm]
\end{figure}

As in the case of $\Gamma=A_3$ above, the next fold has to be at $w$ or the next generator would be a homeomorphism, contradicting that the image of $h$ is a rose, while $A_3$ is not a rose. Since the previous fold was maximal, the next fold cannot be of $\{b'', c''\}$. Also, $\{b'', \bar{d}\}$ and $\{c'', \bar{d}\}$ cannot be illegal turns or $h$ would have had edge backtracking. Thus, $h_i$ was not possible in the first place, meaning that all folds in the Stallings decomposition must be proper full folds between roses, proving (4).

Since all Stallings folds are proper full folds of roses, for each $1 \leq k \leq n-1$, one can index $\mathcal{E}_k = \mathcal{E}(\Gamma_k)$ as $\{E_{(k,1)},\overline{E_{(k,1)}}, E_{(k,2)}, \overline{E_{(k,2)}}, \dots, E_{(k,r)}, \overline{E_{(k,r)}} \} = \{e_{(k,1)}, e_{(k,2)}, \dots, e_{(k,2r-1)}, e_{(k,2r)}\}$ so that \newline
\indent(a) $g_k\colon e_{k-1,j_k} \mapsto e_{k,i_k} e_{k,j_k}$ where $e_{k-1,j_k} \in \mathcal{E}_{k-1}$, $e_{k,i_k}, e_{k,j_k} \in \mathcal{E}_k$ and \newline
\indent (b) $g_k(e_{k-1,i})=e_{k,i}$ for all $e_{k-1,i} \neq e_{k-1,j_k}^{\pm 1}$. \newline
\noindent Suppose we similarly index the directions $D(e_{k,i}) = d_{k,i}$.

Let $g_n=h'$ be the Stallings decomposition's homeomorphism and suppose its edge index permutation were nontrivial. Some power $p$ of the permutation would be trivial. Replace $h$ by $h^p$, rewriting $h^p$'s decomposition as follows. Let $\sigma$ be the permutation defined by $h'(e_{n-1,i})= e_{n-1,\sigma(i)}$ for each $i$. For $n \leq k \leq 2n-p$, define $g_k$ by $g_k: e_{k-1,\sigma^{-s+1}(j_t)} \mapsto e_{k,\sigma^{-s+1}(i_t)} e_{k,\sigma^{-s+1}(j_t)}$ where $k=sp+t$ and $0 \leq t \leq p$. Adjust the corresponding proper full folds accordingly. This decomposition still gives $h^p$, but now the homeomorphism's edge index permutation is trivial, making it unnecessary for the decomposition. \qedhere
\end{proof}

\vskip8pt

Representatives with a decomposition satisfying (I)-(II) of Proposition \ref{P:IdealDecomposition} will be called and \emph{ideally decomposable ($\mathcal{ID}$)} representative with an \emph{ideal decomposition}.

\vskip5pt

\begin{nt}{\label{N:IdealDecompositions}}
\textbf{(Ideal Decompositions)} \newline
We will consider the notation of the proposition standard for an ideal decomposition.  Additionally,
\begin{enumerate}[itemsep=-1.5pt,parsep=3pt,topsep=3pt]
\item We denote $e_{k-1,j_k}$ by $e^{pu}_{k-1}$, denote $e_{k,j_k}$ by $e^u_k$, denote $e_{k,i_k}$ by $e^a_k$, and denote $e_{k-1,i_{k-1}}$ by $e^{pa}_{k-1}$.
\item $\mathcal{D}_k$ will denote the set of directions corresponding to $\mathcal{E}_k$.
\item $f_k:= g_k \circ \cdots \circ g_1 \circ g_n \circ  \cdots \circ g_{k+1}: \Gamma_k \to \Gamma_k$.
\item $$g_{k,i}:=
\begin{cases}
g_k \circ \cdots \circ g_i\colon \Gamma_{i-1} \to \Gamma_k \text{ if $k>i$}\text{ and } \\
 g_k \circ \cdots \circ g_1 \circ g_n \circ  \cdots \circ g_i \text{ if $k<i$}
\end{cases}.$$
\item $d^u_k$ will denote $D_0(e^u_k)$, sometimes called the \emph{\textbf{u}nachieved direction} for $g_k$, as it is not in $Im(Dg_k)$.
\item $d^a_k$ will denote $D_0(e^a_k)$, sometimes called the \emph{twice-\textbf{a}chieved direction} for $g_k$, as it is the image of both $d^{pu}_{k-1}$ ($=D_0(e_{k-1,j_k})$) and $d^{pa}_{k-1}$ ($=D_0(e_{k-1,i_k})$) under $Dg_k$. $d^{pu}_{k-1}$ will sometimes be called the \emph{\textbf{p}re-unachieved direction} for $g_k$ and $d^{pa}_{k-1}$ the \emph{\textbf{p}re-twice-achieved direction} for $g_k$.
\item $G_k$ will denote the ltt structure $G(f_k)$
\item $G_{k,l}$ will denote the subgraph of $G_l$ containing
~\\
\vspace{-6.5mm}
{\begin{itemize}
\item all black edges and vertices (given the same colors and labels as in $G_l$) and \\[-5.5mm]
\item all colored edges representing turns in $g_{k,l}(e)$ for some $e \in \mathcal{E}_{k-1}$. \\[-5.5mm]
\end{itemize}}
\item For any $k,l$, we have a direction map $Dg_{k,l}$ and an induced map of turns $Dg_{k,l}^t$.  The \emph{induced map of ltt Structures} $Dg_{k,l}^T: G_{l-1} \mapsto G_k$ (which we show below exists) is such that
~\\
\vspace{-6.5mm}
{\begin{itemize}
\item the vertex corresponding to a direction $d$ is mapped to the vertex corresponding to $Dg_{k,l}(d)$, \\[-5mm]
\item the colored edge [$d_1, d_2$] is mapped linearly as an extension of the vertex map to the edge [$Dg_{k,l}^t(\{d_1, d_2 \})$] $=$ [$Dg_{k,l}(d_1), Dg_{k,l}(d_2)$], and \\[-5mm]
\item the interior of the black edge of $G_{l-1}$ corresponding to the edge $E \in \mathcal{E}(\Gamma_{l-1})$ is mapped to the interior of the smooth path in $G_k$ corresponding to $g(E)$.
\end{itemize}}

\begin{ex}{\label{Ex:InducedMap}} We describe an induced map of rose-based ltt structures for $g_2: x \mapsto xz$.
~\\
\vspace{-7mm}
\begin{figure}[H]
\centering
\includegraphics[width=2.7in]{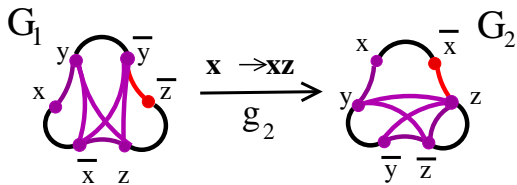}
\hspace{1.7in}\parbox{6.5in}{\caption{\small{\emph{The induced map for $g_2: x \mapsto xz$ sends vertex $\bar{x}$ of $G_1$ to vertex $\bar{z}$ of $G_2$ and every other vertex of $G_1$ to the identically labeled vertex of $G_2$. [$y$] in $G_1$ maps to [$y$] in $G_2$, [$z$] in $G_1$ maps to [$z$] in $G_2$, and [$x$] in $G_1$ maps to $[x] \cup [\bar{x}, z] \cup [z]$ in $G_2$. The purple edge $[\bar{x}, y]$ in $G_1$ maps to the purple edge $[\bar{z}, y]$ in $G_2$, the purple edge $[\bar{x}, \bar{y}]$ in $G_1$ maps to the purple edge $[\bar{z}, \bar{y}]$ in $G_2$, $[\bar{x}, z]$ in $G_1$ maps to the purple edge $[\bar{z}, z]$ in $G_2$, and each other purple edge in $G_1$ is sent to the identically labeled purple edge in $G_2$. The red edge $[\bar{z}, \bar{y}]$ in $G_1$ maps to the purple edge $[\bar{z}, \bar{y}]$ in $G_2$. }}}}
\end{figure}
\end{ex}
\item $\mathcal{C}(G_k)$ will denote the subgraph of $G_k$, coming from $\mathcal{LW}(f_k)$ and containing all colored (red and purple) edges of $G_k$.
\item Sometimes we use $\mathcal{PI}(G_k)$ to denote the purple subgraph of $G_k$ coming from $\mathcal{SW}(f_k)$.
\item $Dg_{k,l}^C$ will denote the restriction (which we show below exists) to $\mathcal{C}(G_{l-1})$ of $Dg_{k,l}^T$.
\item If we additionally require $\phi \in \mathcal{AFI}_r$ and $\mathcal{IW}(\phi) \in \mathcal{PI}_{(r;(\frac{3}{2}-r))}$, then we will say $g$ \emph{has $(r;(\frac{3}{2}-r))$ potential}. (By saying $g$ \emph{has $(r;(\frac{3}{2}-r))$ potential}, it will be implicit that, not only is $\phi \in \mathcal{AFI}_r$, but $\phi$ is ideally decomposed, or at least $\mathcal{ID}$).
\end{enumerate}
\end{nt}

\begin{rk} For typographical clarity, we sometimes put parantheses around subscripts. We refer to $E_{k,i}$ as $E_i$, and $\Gamma_k$ as $\Gamma$, for all $k$ when $k$ is clear.
\end{rk}

\section{Birecurrency Condition}{\label{S:bc}}

Proposition \ref{P:BC} of this section gives a necessary condition for an ideal Whitehead graph to be achieved. We use it to prove Theorem \ref{T:MainTheorem}a, and implicitly throughout this paper and \cite{p12d}.

Definitions of lines and the attracting lamination for a $\phi \in Out(F_r)$ will be as in \cite{bfh00}. A complete summary of relevant definitions can be found in \cite{p12a}. We use \cite{bfh00} that a $\phi \in \mathcal{FI}_r$ has a unique attracting lamination (we denote by $\Lambda_{\phi}$) and that attracting laminations contain birecurrent leaves.

Note that there is both notational and terminology variance in the name assigned to an attracting lamination. It is called a \emph{stable lamination} in \cite{bfh97} and is sometimes also referred to in the literature as an \emph{expanding lamination}. In \cite{bfh97} and \cite{bfh00}, it is denoted $\Lambda^+_{\phi}$, or just $\Lambda^+$, while the authors of \cite{hm11} used the notation $\Lambda_-$, more consistent with dynamical systems terminology.

\begin{df} A \emph{train track (tt) graph} is a finite graph $G$ satisfying:
~\\
\vspace{-\baselineskip}
\begin{description}
\item[tt1:] $G$ has no valence-1 vertices; \\[-6mm]
\item[tt2:] each edge of $G$ has 2 distinct vertices (single edges are never loops); and \\[-6mm]
\item[tt3:] the edge set of $G$ is partitioned into two subsets, $\mathcal{E}_b$ (the ``black'' edges) and $\mathcal{E}_c$ (the ``colored'' edges), such that each vertex is incident to at least one $E_b \in \mathcal{E}_b$ and at least one $E_c \in \mathcal{E}_c$. \\[-6mm]
\end{description}

\noindent tt graphs are \emph{equivalent} that are isomorphic as graphs via an isomorphism preserving the edge partition. And a path in a tt graph is \emph{smooth} that alternates between edges in $\mathcal{E}_b$ and edges in $\mathcal{E}_c$.
\end{df}

\begin{ex}
The ltt structure $G(g)$ for a pNp-free representative $g$ on the rose is a train track graph where the black edges are in $\mathcal{E}_b$ and $\mathcal{E}_c$ is the edge set of $\mathcal{C}(G(g))$.
\end{ex}

\begin{df} A smooth tt graph is \emph{birecurrent} if it has a locally smoothly embedded line crossing each edge infinitely many times as $\bold{R}\to \infty$ and as $\bold{R}\to -\infty$. \end{df}

\begin{prop}{\label{P:BC}}\textbf{(Birecurrency Condition)} The lamination train track structure for each train track representative of each fully irreducible outer automorphism is birecurrent. \end{prop}

Our proof requires the following lemmas relating $\mathcal{LW}(g)$ and realization of leaves of $\Lambda_{\phi}$. The proofs use lamination facts from \cite{bfh97} and \cite{hm11}.

\begin{lem}{\label{L:LeafTurns}} Let $g: \Gamma \to \Gamma$, with $(r;(\frac{3}{2}-r))$ potential, represent $\phi \in Out(F_r)$. The only possible turns taken by the realization in $\Gamma$ of a leaf of $\Lambda_{\phi}$ are those giving edges in $\mathcal{LW}(g)$. Conversely, each turn represented by an edge of $\mathcal{LW}(g)$ is a turn taken by some (hence all) leaves of $\Lambda_{\phi}$ (as realized in $\Gamma$).
\end{lem}

\begin{proof} First note that, since $g$ is irreducible, each $E_i \in \mathcal{E}(\Gamma)$ has an interior fixed point. Thus, for each $E_i \in \mathcal{E}(\Gamma)$, there is a periodic leaf of $\Lambda_{\phi}$ obtained by iterating a neighborhood of a fixed point of $E_i$.

Consider any turn $\{d_1, d_2\}$ taken by the realization in $\Gamma$ of a leaf of $\Lambda_{\phi}$. Since periodic leaves are dense in the lamination, each periodic leaf of the lamination contains a subpath taking the turn. In particular, the leaf obtained by iterating a neighborhood of a fixed point of $e$ for any $e \in \mathcal{E}(\Gamma)$ takes the turn, so $\overline{e_1} e_2$ (where $D_0(e_1)=d_1$ or $D_0(e_2)=d_2$) is contained in some $g^k(e)$, for each $e \in \mathcal{E}(\Gamma)$.  So $\{d_1, d_2\}$ is represented by an edge in $\mathcal{LW}(g)$, concluding the forward direction.

If $[d_1, d_2]$ is an edge of $\mathcal{LW}(g)$ then, for some $i$ and $k$, $\overline{e_1} e_2$ is a subpath of $g^k(E_i)$. Again, each $E_i \in \mathcal{E}(\Gamma)$ has an interior fixed point and hence $\Lambda_{\phi}$ has a periodic leaf obtained by iterating a neighborhood of $E_i$'s fixed point. $g^k(E_i)$ is a subpath of this periodic leaf and (by periodic leaf density) of every leaf of $\Lambda_{\phi}$. Since the leaves contain $g^k(E_i)$ as a subpath, they contain $\overline{e_1} e_2$ as a subpath, so $\{d_1, d_2\}$. \qedhere
\end{proof}

\begin{lem}{\label{L:SmoothPathsforLeaves}} Let $g\colon\Gamma \to \Gamma$ represent $\phi \in \mathcal{AFI}_r$.  Then $G(g)$ contains a smooth path corresponding to the realization in $\Gamma$ of each leaf of $\Lambda_{\phi}$. \end{lem}

\begin{proof} Consider the realization $\lambda$ of a leaf of $\Lambda_{\phi}$ and any single subpath $\sigma=e_1 e_2 e_3$ in $\lambda$. If it exists, the representation in $G(g)$ of $\sigma$ would be the path $[d_1, \overline{d_1}, d_2, \overline{d_2}, d_3, \overline{d_3}]$. Lemma \ref{L:LeafTurns} tells us $[\overline{d}_1, d_2]$ and $[\overline{d_2}, d_3]$ are edges of $\mathcal{LW}(g)$, hence are in $\mathcal{C}(G(g))$. The path representing $\sigma$ in $G(g)$ thus exists and alternates between colored and black edges. Analyzing overlapping subpaths to verifies smoothness. \qedhere
\end{proof}

\begin{proof}[Proof of Proposition \ref{P:BC}] We show that the path $\gamma$ corresponding to the realization $\lambda$ of a leaf of $\Lambda_{\phi}$ is a locally smoothly embedded line in $G(g)$ traversing each edge of $G(g)$ infinitely many times as $\bold{R}\to \infty$ and as $\bold{R}\to -\infty$.  By Lemma \ref{L:LeafTurns}, for any colored edge [$d_i, d_j$] in $G(g)$, $\lambda$ must contain either $\overline{e_i} e_j$ or $\overline{e_2} e_1$ as a subpath.  Fully irreducible outer automorphism lamination leaf birecurrency implies $\gamma$ must traverse the subpath $\overline{e_i} e_j$ or $\overline{e_j} e_i$ infinitely many times as $\bold{R}\to \infty$ and as $\bold{R}\to -\infty$.  By Lemma \ref{L:SmoothPathsforLeaves}, this concludes the proof for a colored edge. Consider a black edge $[d_l, \overline{d_l}]=[e_l]$.  Each vertex is shared with a colored edge.  Let [$d_l, \overline{d_m}$] be such an edge.  As shown above, $\overline{e_l} e_m$ or $\overline{e_m} e_l$ occur in a realization $\lambda$ infinitely many times as $\bold{R}\to \infty$ and as $\bold{R}\to -\infty$.  So $\lambda$ traverses $e_l$ infinitely many times as $\bold{R}\to \infty$ and as $\bold{R}\to -\infty$. Thus $\gamma$ traverses $[e_l]$ infinitely many times as $\bold{R}\to \infty$ and as $\bold{R}\to -\infty$. \qedhere
\end{proof}

\section{Admissible map properties}{\label{Ch:AMProperties}}

We prove that the ideal decomposition of a $(r;(\frac{3}{2}-r))$ potential representative satisfies ``Admissible Map Properties'' listed in Proposition \ref{P:am}. In Section \ref{Ch:Peels} we use the properties to show there are only two possible (fold/peel) relationship types between adjacent ltt structures in an ideal decomposition. Using this, in Section \ref{Ch:AMDiagrams}, we define the ``ideal decomposition diagram'' for $\mathcal{G} \in \mathcal{PI}_{(r;(\frac{3}{2}-r))}$.

The statement of Proposition \ref{P:am} comes at the start of this section, while its proof comes after a sequence of technical lemmas used in the proof.

\emph{\textbf{$g:\Gamma \to \Gamma$ will represent $\phi\in Out(F_r)$, have $(r;(\frac{3}{2}-r))$ potential, and be ideally decomposed as: $\Gamma = \Gamma_0 \xrightarrow{g_1} \Gamma_1 \xrightarrow{g_2} \cdots \xrightarrow{g_{n-1}}\Gamma_{n-1} \xrightarrow{g_n} \Gamma_n = \Gamma$. We use the standard \ref{N:IdealDecompositions} notation.}} \newline

\begin{prop}{\label{P:am}} $g$ satisfies each of the following.
\begin{description}
\item[AM Property I:] Each $G_j$ is birecurrent. \\[-5mm]
\item[AM Property II:] For each $G_j$, the illegal turn $T_j$ for the generator $g_{j+1}$ exiting $G_j$ contains the unachieved direction $d^u_j$ for the generator $g_j$ entering $G_j$, i.e. either $d^u_j=d^{pa}_j$ or $d^u_j=d^{pu}_j$. \\[-5mm]
\item[AM Property III:] In each $G_j$, the vertex labeled $d^u_j$ and edge $[t^R_j]=[d^u_j, \overline{d^a_j}]$ are both red. \\[-5mm]
\item[AM Property IV:]  If $[d_{(j,i)}, d_{(j,l)}]$ is in $C(G_j)$, then $D^Cg_{m,j+1}$([$d_{(j,i)}, d_{(j,l)}$]) is a purple edge in $G_m$, for each $m \neq j$. \\[-5mm]
\item[AM Property V:] For each $j$, $[t^R_j]=[d^u_j, \overline{d^a_j}]$ is the unique edge containing $d^u_j$. \\[-5mm]
\item[AM Property VI:] Each $g_j$ is defined by $g_j:e^{pu}_{j-1} \mapsto e^a_j e^u_j$ (where $D_0(e^u_j)=d^u_j$, $D_0(\overline{e^a_j})=\overline{d^a_j}$, $e^u_j=e_{j,m}$, and $e^{pu}_{j-1}=e_{j-1,m}$). \\[-5mm]
\item[AM Property VII:] $Dg_{l,j+1}$ induces an isomorphism from $SW(f_j)$ onto $SW(f_l)$ for all $j \neq l$. \\[-5mm]
\item[AM Property VIII:] For each $1 \leq j \leq r$:
~\\
\vspace{-\baselineskip}
\begin{itemize}
\item[(a)] there exists a $k$ such that either $e^u_k=E_{k,j}$ or $e^u_k= \overline{E_{k,j}}$ and \\[-5mm]
\item[(b)] there exists a $k$ such that either $e^a_k=E_{k,j}$ or $e^a_k= \overline{E_{k,j}}$.
\end{itemize}
\end{description}
\end{prop}

\vskip8pt

\noindent The proof of Proposition \ref{P:am} will come at the end of this subsection.

\vskip8pt

\begin{df}
An edge path $\gamma=e_1 \dots e_k$ in $\Gamma$ has \emph{cancellation} if $\overline{e_i} =e_{i+1}$ for some $1 \leq i \leq k-1$. We say $g$ has \emph{no cancellation on edges} if for no $l>0$ and edge $e \in \mathcal{E}(\Gamma)$ does $g^l(e)$ have cancellation.
\end{df}

\begin{lem}{\label{L:PreLemma}} For this lemma we index the generators in the decomposition of all powers $g^p$ of $g$ so that $g^p=g_{pn} \circ g_{pn-1} \circ \dots \circ g_{(p-1)n} \circ \dots \circ g_{(p-2)n} \circ \dots \circ g_{n+1} \circ g_n \circ \dots \circ g_1$ ($g_{mn+i}=g_i$, but we want to use the indices to keep track of a generator's place in the decomposition of $g^p$).  With this notation, $g_{k,l}$ will mean $g_k \circ \dots \circ g_l$.  Then: \newline
\indent 1. for each $e \in \mathcal{E}(\Gamma_{l-1}$), no $g_{k,l}(e)$ has cancellation; \newline
\indent 2. for each $0 \leq l \leq k$ and $E_{l-1,i} \in \mathcal{E}^+(\Gamma_{l-1})$, the edge $E_{k,i}$ is in the path $g_{k,l} (E_{l-1,i})$; and \newline
\indent 3. if $e^u_k=e_{k,j}$, then the turn $\{\overline{d^a_k}, d^u_k \}$ is in the edge path $g_{k,l}(e_{l-1,j})$, for all $0 \leq l \leq k$.
\end{lem}

\begin{proof} Let $s$ be minimal so that some $g_{s,t} (e_{t-1,j})$ has cancellation. Before continuing with our proof of (1), we first proceed by induction on $k-l$ to show that (2) holds for $k<s$.  For the base case observe that $g_{l+1}(e_{l,j})=e_{l+1,j}$ for all $e_{l+1,j} \neq (e^{pu}_l)^{\pm 1}$. Thus, if $e_{l,j} \neq e^{pu}_l$ and $e_{l,j} \neq \overline{e^{pu}_l}$ then $g_{l+1}(e_{l,j})$ is precisely the path $e_{l+1,j}$ and so we are only left for the base case to consider when $e_{l,j} = (e^{pu}_l)^{\pm 1}$.  If $e_{l,j} = e^{pu}_l$, then $g_{l+1}(e_{l,j})=e^a_{l+1} e_{l+1,j}$ and so the edge path $g_{l+1}(e_{l,j})$ contains $e_{l+1,j}$, as desired.  If $e_{l,j} =\overline{e^{pu}_l}$, then $g_{l+1}(e_{l,j})=e_{l+1,j} \overline{e^a_{l+1}}$ and so the edge path $g_{l+1}(e_{l,j})$ also contains $e_{l+1,j}$ in this case.  Having considered all possibilities, the base case is proved.

For the inductive step, we assume $g_{k-1,l+1} (e_{l,j})$ contains $e_{k-1,j}$ and show $e_{k,j}$ is in the path $g_{k,l+1}(e_{l,j})$. Let $g_{k-1,l+1}(e_{l,j})= e_{i_1}\dots e_{i_{q-1}} e_{k-1,j} e_{i_{q+1}}\dots e_{i_r}$ for some edges $e_i \in \mathcal{E}_{k-1}$. As in the base case, for all $e_{k-1,j} \neq (e^u_k)^{\pm 1}$, $g_k(e_{k-1,j})$ is precisely the path $e_{k,j}$. Thus (since $g_k$ is an automorphism and since there is no cancellation in $g_ {j_1,j_2}(e_{j_1,j_2})$ for $1 \leq j_1 \leq j_2 \leq k$), $g_{k,l+1} (e_{l,j})=\gamma_1 \dots \gamma_{q-1}  (e_{k,j}) \gamma_{q+1} \dots \gamma_m$ where each $\gamma_{i_j}= g_l({e_{i_j}})$ and where no $\{\overline{\gamma_i}, \gamma_{i+1} \}$, $\{\overline{e_{k,j}}, \gamma_{q+1} \}$, or $\{\overline{\gamma_{q-1}}, e_{k,j} \}$ is an illegal turn.
So each $e_{k,j}$ is in $g_{k,l+1} (e_{l,j})$. We are only left to consider for the inductive step the cases $e_{k-1,j}= e^{pu}_k$ and $e_{k-1,j} =\overline {e^{pu}_k}$.

If $e_{k-1,j} = e^{pu}_k$, then $g_k(e_{k-1,j})=e^a_k e_{k,j}$, and so $g_{k,l+1}(e_{l,j}) = \gamma_1 \dots \gamma_{q-1} e^a_k e_{k,j} \gamma_{q+1} \dots \gamma_m$ (where no $\{\overline{\gamma_i}, \gamma_{i+1} \}$, $\{\overline{e_{k,j}}, \gamma_{q+1} \}$, or $\{\overline{\gamma_{q-1}}, e^a_k \}$ is an illegal turn), which contains $e_{k,j}$, as desired. If instead $e_{k-1,j} =\overline {e^{pu}_k}$, then $g_k(e_{k-1,j})=e_{k,j} \overline{e^a_k}$ and so $g_{k,l+1}(e_{l,j})=\gamma_1 \dots \gamma_{q-1} e_{k,j} \overline{e^a_k} \gamma_{q+1} \dots\gamma_m$, which also contains $e_{k,j}$. Having considered all possibilities, the inductive step is now also proven and the proof is complete for (2) in the case of $k<s$.

We finish the proof of (1). $s$ is still minimal. So $g_{s,t}(e_{t-1,j})$ has cancellation for some $e_{t-1,j} \in \mathcal{E}_j$. Suppose $g_{s,t}(e_{t-1,j})$ has cancellation. For $1 \leq j \leq m$, let $\alpha_j \in \mathcal{E}_{s-1}$ be such that $g_{s-1,t}(e_{t-1,j})= \alpha_1 \cdots \alpha_m$. By $s$'s minimality, either $g_s(\alpha_i)$ has cancellation for some $1 \leq i \leq m$ or $Dg_s(\overline{\alpha_i})=Dg_s(\alpha_{i+1})$ for some $1 \leq i < m$. Since each $g_s$ is a generator, no $g_s(\alpha_i)$ has cancellation. So, for some $i$, $Dg_s(\overline{\alpha_i})= Dg_s(\alpha_{i+1})$. As we have proved (1) for all $k<s$, we know $g_{t-1,1}(e_{0,j})$ contains $e_{t-1,j}$. So $g_{s,1}(e_{0,j})= g_{s,t}(g_{t-1,1}(e_{0,j}))$ contains cancellation, implying $g^p(e_{0,j})= g_{pn,s+1}(g_{s,1}(e_{0,j}))= g_{s,t}(\dots e_{t-1,j} \dots)$ for some $p$ (with $pn>s+1$) contains cancellation, contradicting that $g$ is a train track.

We now prove (3). Let $e^u_k=e_{k,l}$. By (2) we know that the edge path $g_{k-1,l}(e_{l-1,j})$ contains $e_{k-1,j}$.  Let $e_1, \dots e_m \in \mathcal{E}_{k-1}$ be such that $g_{k-1,l}(e_{l-1,j})=e_1 \dots e_{q-1} e_{k-1,j} e_{q+1} \dots e_m$. Then $g_{k,l}(e_{l-1,j})=\gamma_1 \dots \gamma_{q-1} e^a_k e^u_k \gamma_{q+1} \dots \gamma_r$ where $\gamma_j=g_k(e_j)$ for all $j$. Thus $g_{k,l}(e^{pu}_{k-1})$ contains $\{ \overline{d^a_k}, d^u_k \}$, as desired.  \qedhere
\end{proof}

\vskip5pt

\begin{lem}{\label{L:fk}} (Properties of $f_k= g_k \circ g_{k-1} \circ \cdots \circ g_{k+2} \circ g_{k+1}\colon\Gamma_k\to\Gamma_k$)
~\\
\vspace{-\baselineskip}
\begin{description}
\item[a.] Each $f_k$ represents the same $\phi$.  In particular, if $g$ has $(r;(\frac{3}{2}-r))$ potential, then so does each $f_k$. \\[-5.7mm]
\item[b.] Each $f_k$ is rotationless.  In particular, all periodic directions are fixed. \\[-5.7mm]
\item[c.] Each $f_k$ has 2r-1 gates (and thus periodic directions). \\[-5.7mm]
\item[d.] For each $k$, $d^u_k \notin \mathcal{IM}(Df_k)$. Thus, $d^u_k$ is the unique nonperiodic (in fact nonfixed) direction for $Df_k$. \\[-5mm]
\item[e.] If $\Gamma = \Gamma_0 \xrightarrow{g_1} \Gamma_1 \xrightarrow{g_2} \cdots \xrightarrow{g_{n-1}}\Gamma_{n-1} \xrightarrow{g_n} \Gamma_n = \Gamma$ is an ideal decomposition of $g$, then $\Gamma_k \xrightarrow{g_{k+1}} \Gamma_{k+1} \xrightarrow{g_{k+2}} \cdots \xrightarrow{g_{k-1}} \Gamma_{k-1} \xrightarrow{g_k} \Gamma_k$ is an ideal decomposition of $f_k$.
\end{description}
\end{lem}

\begin{proof} Lemma \ref{L:pNpFreePreserved} implies (a). Each $f_k$ is rotationless, as it represents a rotationless $\phi$. This gives (b). We prove (c). The number of gates is the number of periodic directions, which here (by (b)) is the number of fixed directions. $f_k$ is on the rose, so has a single local Stable Whitehead graph. Lemma \ref{L:pNpFreePreserved} implies $f_k$, as $g$, has no pNps. So $\mathcal{SW}(f_k) \cong \mathcal{IW}(\phi)$, which has 2r-1 vertices. So $f_k$ has 2r-1 periodic directions, thus gates. We prove (d). By (b) and (c), $Df_k$ has 2r-1 fixed directions. Since $d^u_k \notin \mathcal{IM}(Dg_k)$, it cannot be in $\mathcal{IM}(Df_k)$, so is the unique nonfixed direction. We prove (e). Ideal decomposition properties (I)-(IIb) hold for $f_k$'s decomposition, as they hold for $g$'s decomposition and the decompositions have the same $\Gamma_i$ and $g_i$ (renumbered). (IIc) holds for $f_k$'s decomposition by (d).  \qedhere
\end{proof}

\noindent \emph{\textbf{We add to the notation already established: $t^R_k=\{\overline{d^a_k},d^u_k\}$, $e^R_k=[t^R_k]$, and $T_k=\{d^{pa}_k, d^{pu}_k \}$.}}

\smallskip

\begin{lem}{\label{L:IllegalTurn}} The following hold for each $T_k=\{d^{pa}_k, d^{pu}_k \}$.
~\\
\vspace{-6mm}
\begin{description}
\item[a.] $T_k$ is an illegal turn for $g_{k+1}$ and, thus, also for $f_k$. \\[-5.7mm]
\item[b.] For each $k$, $T_k$ contains $d^{u}_{k}$. \\[-5mm]
\end{description}
\end{lem}

\begin{proof} Recall that $T_k=\{d^{pa}_k, d^{pu}_k \}$.  Since $D^tg_{k+1}(\{d^{pa}_k, d^{pu}_k \})=\{Dg_{k+1}(d^{pa}_k), Dg_k(d^{pu}_k) \}= \{d^a_{k+1}, d^a_{k+1} \}$, \newline
\noindent  $D^tf_k(\{d^{pa}_k, d^{pu}_k\})=D^t(g_{k,k+2} \circ g_{k+1}) (\{d^{pa}_k, d^{pu}_k\})= D^t(g_{k,k+2}) (D^tg_{k+1} (\{d^{pa}_k, d^{pu}_k\}))= D^tg_{k,k+2} (\{d^a_{k+1}, d^a_{k+1}\})$ \newline
\noindent  $=\{D^tg_{k,k+2}(d^a_{k+1}), D^tg_{k,k+2}(d^a_{k+1})\}$, which is degenerate.  So $T_k$ is an illegal turn for $f_k$, proving (a).

For (b) suppose $g$ has $2r-1$ periodic directions and, for contradiction's sake, the illegal turn $T_k$ does not contain $d^u_k=d_{k,i}$. Let $d^u_{k+1}=d_{k+1,s}$ and $d^a_{k+1}=d_{k+1,t}$. Then $Dg_k(d_{k-1,s})=d_{k,s}$ and $Dg_k(d_{k-1,t})=d_{k,t}$, so $D^t(g_{k+1} \circ g_k)(\{d_{(k-1,s)}, d_{(k-1,t)} \})= \{D(g_{k+1} \circ g_k)(d_{(k-1,s)}), D(g_{k+1} \circ g_k)(d_{(k-1,t)})\}= \{ Dg_{k+1}(d_{k,s}= d^{pu}_k), Dg_{k+1}(d_{k,t}= d^{pa}_k)\} = \{ d^a_{k+1}, d^a_{k+1} \}$. So $d_{k-1,s}$ and $d_{k-1,t}$ share a gate. But $d_{k-1,i}$ already shares a gate with another element and we already established that $d_{k-1,i} \neq d_{k-1,s}$ and $d_{k-1,i} \neq d_{k-1,t}$. So $f_{k-1}$ has at most $2r-2$ gates. Since each $f_k$ has the same number of gates, this implies $g$ has at most $2r-2$ gates, giving a contradiction. (b) is proved.  \qedhere
\end{proof}

\begin{cor}{\label{C:UnachievedDirection}} \textbf{(of Lemma \ref{L:IllegalTurn})} For each $1 \leq k \leq n$,
~\\
\vspace{-\baselineskip}
\begin{description}
\item[a.] $t^R_k= \{\overline{d^a_k}, d^u_k\}$, must contain either $d^{pu}_k$ or $d^{pa}_k$ and \\[-5.2mm]
\item[b.] The vertex labeled $d^u_k$ in $G_k$ is red and $[t^R_k]= [\overline{d^a_k}, d^u_k]$ is a red edge in $G_k$.
\end{description}
\end{cor}

\begin{proof} We start with (a). Lemma \ref{L:IllegalTurn} implies each $T_k$ contains $d^u_k$. At the same time, we know $t^R_k= \{\overline{d^a_k}, d^u_k\}$, implying $t^R_k$ contains $d^u_k$, thus either $d^{pa}_k$ or $d^{pu}_k$. We now prove (b). By Lemma \ref{L:fk}d, $d^u_k$ is not a periodic direction for $Df_k$, so is not a vertex of $\mathcal{SW}(f_k)$. Thus, $d^u_k$ labels a red vertex in $G_k$. To show $[t^R_k]$ is in $\mathcal{LW}(f_k)$ it suffices to show $t^R_k$ is in $f_k(e^u_k)$.  Let $e^u_k=e_{k,l}$.  By Lemma \ref{L:PreLemma}, the path $g_{k-1,k+1} (e^u_k=e_{k,l})$ contains $e_{k-1,l}$.  Let $e_j \in \mathcal{E}_{l-1}$ be such that $g_{k-1,k+1}(e^u_k)= e_1 \dots e_{q-1} e_{k-1,l} e_{q+1} \dots e_m$. Then $f_k(e^u_k ) =g_{k,k+1} (e^u_k )= \gamma_1 \dots \gamma_{q-1} e^a_k e^u_k \gamma_{q+1}\dots \gamma_m$ where $\gamma_j =g_k(e_{i_j})$ for all $j$.  So $f_k(e^u_k)$ contains $\{ \bar d^a_k, d^u_k \}$ and $\mathcal{LW}(f_k)$ contains [$t^R_k$]. Since $[\overline{d^a_k}, d^u_k]$ contains the red vertex $d^u_k$, it is red in $G_k$. \qedhere
\end{proof}

\begin{lem}{\label{L:EdgeImage}} If $[d_{(l,i)},d_{(l,j)}]$ is in $\mathcal{C}(G_l)$, then $[D^tg_{k,l+1}(\{d_{(l,i)},d_{(l,j)}\})]$ is a purple edge in $G_k$.
\end{lem}

\begin{proof}  It suffices to show two things: \newline
\indent (1) $D^tg_{k,l+1}(\{d_{(l,i)},d_{(l,j)}\})$ is a turn in some edge path $f_l^p(e_{l,m})$ with $p \geq 1$ and \newline
\indent (2) $Dg_{k,l+1}(d_{l,i})$ and $Dg_{k,l+1}(d_{l,j})$ are periodic directions for $f_l$. \newline
\noindent We use induction. Start with (1). For the base case assume $[d_{(k-1,i)}, d_{(k-1,j)}]$ is in $\mathcal{C}(G_{k-1})$, so $f_{k-1}^p(e_{k-1,t})= s_1 \dots \overline{e_{(k-1,i)}} e_{(k-1,j)} \dots s_m$ for some $e_{(k-1,t)},s_1, \dots s_m \in \mathcal{E}_{k-1}$ and $p \geq 1$. By Lemma \ref{L:PreLemma}, $e_{k-1,t}$ is in the path $g_{k-1} \circ \cdots \circ g_1 \circ g_n \circ \cdots \circ g_{k+1} (e_{k,t})$. Thus, since $f_{k-1}^p(e_{k-1,t})=s_1 \dots \overline{e_{(k-1,i)}} e_{(k-1,j)} \dots s_m$ and no $g_{i,j}(e_{j-1,t})$ can have cancellation, $s_1 \dots \overline{e_{(k-1,i)}} e_{(k-1,j)} \dots s_m$ is a subpath of $f_{k-1}^p \circ g_{k-1} \circ \cdots \circ g_1 \circ g_n \circ \cdots \circ g_{k+1} (e_{k,t})$. Apply $g_k$ to $f_{k-1}^p \circ g_{k-1} \circ \cdots \circ g_1 \circ g_n \circ \cdots \circ g_k (e_{k-1,t})$ to get $f_k^{p+1}(e_{k,t})$.

Suppose $Dg_k(e_{k-1,i})=e_{k,i}$ and $Dg_k(e_{k-1,j})=e_{k,j}$. Then $g_k(\dots\overline{e_{k-1,i}} e_{k-1,j} \dots) = \dots \overline{e_{(k,i)}} e_{(k,j)} \dots$, with possibly different edges before and after $\overline{e_{k,i}}$ and $e_{k,j}$ than before and after $\overline{e_{k-1,i}}$ and $e_{k-1,j}$. Thus, here, $f^{p+1}_k( \dots \overline{e_{(k-1,i)}} e_{(k-1,j)} \dots)$ contains $\{d_{(k,i)},d_{(k,j)}\}$, which here is $D^tg_k(\{d_{(k-1,i)}, d_{(k-1,j)}\})$. So [$D^tg_k(\{d_{(k-1,i)}, d_{(k-1,j)}\})$] is an edge in $G_k$.

Suppose $g_k\colon e_{k-1,j} \mapsto e_{k,l} e_{k,j}$. Then $g_k(\dots \overline{e_{k-1,i}} e_{k-1,j}\dots)= \dots \overline{e_{k,i}} e_{k,l}e_{k,j} \dots$,
(again with possibly different edges before and after $\overline{e_{k,i}}$ and $e_{k,j}$). So $g_k(\dots \overline{e_{(k-1,i)}} e_{(k-1,j)} \dots)$ contains $\{\overline{d_{(k,l)}},d_{(k,j)}\}$, which here is $D^tg_k(\{d_{(k-1,i)},d_{(k-1,j)}\})$, so [$D^tg_k(\{d_{(k-1,i)},d_{(k-1,j)}\})$] again is in $G_k$.

Finally, suppose $g_k: e_{k-1,j} \mapsto e_{k,j}e_{k,l}$ defined $g_k$. Unless $\overline{e_{k-1,i}} = e_{(k-1,j)}$, we have \newline \noindent $g_k(\dots \overline{e_{(k-1,i)}} e_{(k-1,j)} \dots)=\dots \overline{e_{(k,i)}} e_{(k,j)} e_{(k,l)} \dots$, containing $\{d_{(k,i)},d_{(k,j)}\}= D^tg_k(\{d_{(k-1,i)},d_{(k-1,j)}\})$. So [$D^tg_k(\{d_{(k-1,i)},d_{(k-1,j)}\})$] is an edge in $G_k$ here also.

If $\overline{e_{k-1,i}} = e_{k-1,j}$, we are in a reflection of the previous case. The other cases ($g_k: \overline{e_{k-1,i}} \mapsto \overline{e_{k,i}} e_{k,l}$ and $g_k: \overline{e_{k-1,i}} \mapsto e_{k,l} \overline{e_{k,i}}$) follow similarly by symmetry. The base case for (1) is complete.

We prove the base case of (2). Since $[D^tg_k(\{d_{(k-1,i)},d_{(k-1,j)}\})]= [Dg_k(d_{(k-1,i)}),Dg_k(d_{(k-1,j)})]$, both vertex labels of $[D^tg_k(\{d_{(k-1,i)},d_{(k-1,j)}\})]$ are in $\mathcal{IM}(Dg_k)$. By Lemma \ref{L:fk}d, this means both vertices are periodic. So $[D^tg_k(\{d_{(k-1,i)},d_{(k-1,j)}\})]$ is in $\mathcal{PI}(G_k)$. The base case is proved. Suppose inductively $[d_{(l,i)},d_{(l,j)}]$ is an edge in $\mathcal{C}(G_l)$ and $[D^tg_{k-1,l+1}(\{d_{(l,i)},d_{(l,j)}\})]$ is an edge in $\mathcal{PI}(G_{k-1})$. The base case implies $[D^tg_k(D^tg_{k-1,l+1}(\{d_{(l,i)},d_{(l,j)}\})]$ is an edge in $\mathcal{PI}(G_k)$. But $D^tg_k(D^tg_{k-1,l+1}(\{d_{(l,i)},d_{(l,j)}\}))=$ \newline \noindent $D^tg_{k,l+1}(\{d_{(l,i)},d_{(l,j)}\})$. The lemma is proved. \qedhere
\end{proof}

\begin{lem}{\label{L:RedEdgeImage}} (Properties of $t^R_k$ and $e^R_k$). For each $1 \leq l,k \leq n$
~\\
\vspace{-5.5mm}
\begin{description}
\item[a.] $[D^tg_{l,k}(\{\overline{d^a_{k-1}}, d^u_{k-1} \})]$ is a purple edge in $G_l$.\\[-5mm]
\item[b.] $[\overline{d^a_k}, d^u_k]$ is not in $D^Cg_k(G_{k-1})$.
\end{description}
\end{lem}

\begin{proof} By Lemma \ref{L:EdgeImage}, it suffices to show for (a) that $[\overline{d^a_{k-1}}, d^u_{k-1}]$ is a colored edge of $G_{k-1}$. This was shown in Corollary \ref{C:UnachievedDirection}b. By Lemma \ref{L:EdgeImage}, each edge in $\mathcal{C}(G_{k-1})$ is mapped to a purple edge in $G_k$. On the other hand, $[\overline{d^a_k}, d^u_k]$ is a red edge in $G_k$. Thus, $[\overline{d^a_k}, d^u_k]$ is not in $D^Cg_k(G_{k-1})$ and (b) is proved.  \qedhere
\end{proof}

\vskip10pt

\noindent Each $G_k$ has a unique red edge ($e^R_k=[t^R_k]=[\overline{d^a_k},d^u_k]$):

\begin{lem}{\label{L:1Edge}} $\mathcal{C}(G_k)$ can have at most 1 edge segment connecting the nonperiodic direction red vertex $d^u_k$ to the set of purple periodic direction vertices. \end{lem}

\begin{proof}  First note that the nonperiodic direction $d^u_k$ labels the red vertex in $G_k$. If $g_k(e_{k-1,i})= e_{k,i}e_{k,j}$, then the red vertex in $G_k$ is $\overline{d_{k,i}}$ (where $d_{k,i}=D_0(e_{k,i})$ and $d_{k,j}=D_0(e_{k,j})$). The vertex $\overline{d_{k,i}}$ will be adjoined to the vertex for $d_{k,j}$ and only $d_{k,j}$: each occurrence of $e_{k-1,i}$ in the image under $g_{k-1,1}$ of any edge has been replaced by $e_{k,i}e_{k,j}$ and every occurrence of $\overline{e_{k,i}}$ has been replaced by $\overline{e_{k,i}} \overline{e_{k,j}}$, ie, there are no copies of $e_{k,j}$ without $e_{k,i}$ following them and no copies of $\overline{e_{k,i}}$ without $\overline{e_{k,j}}$ preceding them. \qedhere
\end{proof}

\vskip10pt

\noindent The red edge and vertex of $G_k$ determine $g_k$:

\begin{lem}{\label{L:NG}} Suppose that the unique red edge in $G_k$ is $[t^R_k]= [d_{(k,j)}, \overline{d_{(k,i)}}]$ and that the vertex representing $d_{k,j}$ is red. Then $g_k(e_{k-1,j}) = e_{k,i} e_{k,j}$ and $g_k(e_{k-1,t})=e_{k,t}$ for $e_{k-1,t} \neq (e_{k-1,j})^{\pm 1}$, where $D_0(e_{s,t})=d_{s,t}$ and $D_0(\overline{e_{s,t}}) = \overline{d_{s,t}}$ for all $s$, $t$.  \end{lem}

\begin{proof}  By the ideal decomposition definition, $g_k$ is defined by $g_k: e_{k-1,j} \mapsto e_{k,i} e_{k,j}$. Corollary \ref{C:UnachievedDirection} implies $D_0(e_{k,j})=d^u_k$, i.e. the direction associated to the red vertex of $G_k$. So the second index of $d^u_k$ uniquely determines the index $j$, so $e_{k-1,j}= e^{pu}_{k-1}$ and $e_{k,i}=e^a_k$.  Additionally, Corollary \ref{C:UnachievedDirection}'s proof implies $[\overline{d_{(k,i)}}, d_{(k,j)}]$ is $G_k$'s red edge. So $e_{k,i}=e^a_k$. And $g_k$ must be $g_k : e^{pu}_{k-1} \mapsto e^a_k e^u_k$, i.e, $e_{k-1,j} \mapsto e_{k,i} e_{k,j}$. \qedhere
\end{proof}

\begin{lem}{\label{L:PurpleEdgeImages}} (Induced maps of ltt structures)
~\\
\vspace{-5.7mm}
\begin{description}
\item[a.] $D^Cf_k$ maps $\mathcal{PI}(G_k)$ isomorphically onto itself via a label-preserving isomorphism. \\[-5.8mm]
\item[b.] The set of purple edges of $G_{k-1}$ is mapped by $D^Cg_k$ injectively into the set of purple edges of $G_k$. \\[-5.8mm]
\item[c.] For each $0 \leq l,k \leq n$, $Dg_{l, k+1}$ induces an isomorphism from $\mathcal{SW}(f_k)$ onto $\mathcal{SW}(f_l)$. \\[-5mm]
\end{description}
\end{lem}

\begin{proof} We prove (a). Lemma \ref{L:EdgeImage} implies that $D^Cf_k$ maps $\mathcal{PI}(G_k)$ into itself. However, $Df_k$ fixes all directions labeling vertices of $\mathcal{SW}(f_k)=\mathcal{PI}(G_k)$. Thus, $D^Cf_k$ restricted to $\mathcal{PI}(G_k)$, is a label-preserving graph isomorphism onto its image.

We prove (b). Since $d^a_k$ is the only direction with more than one $Dg_k$ preimage, and these two preimages are $d^{pa}_{k-1}$ and $d^{pu}_{k-1}$, the $[d_{(k,i)}, d^a_k]$ are the only edges in $G_k$ with more than one $D^Cg_k$ preimage. The two preimages are the edges $[d_{(k-1,i)}, d^{pa}_{k-1}]$ and $[d_{(k-1,i)}, d^{pu}_{k-1}]$ in $G_{k-1}$. However, by Lemma \ref{L:IllegalTurn}, either $e^u_{k-1}= e^{pu}_{k-1}$ or $e^u_{k-1}=e^{pa}_{k-1}$. So one of the preimages of $d^a_k$ is actually $d^u_{k-1}$, i.e. one of the preimage edges is actually $[d_{(k-1,i)}, d^u_{k-1}]$. Since $[t^R_{k-1}]$ is the only edge of $C(G_{k-1})$ containing $d^u_{k-1}$, one of the preimages of $[d_{(k,i)}, d^a_k]$ must be $[t^R_{k-1}]$, leaving only one possible purple preimage.

We prove (c). By (b), the set of $G_k$'s purple edges is mapped injectively by $D^Cg_{l, k+1}$ into the set of $G_l$'s purple edges. Likewise, the set of $G_l$'s purple edges is mapped injectively by $D^Cg_{k, l+1}$ into $G_k$. (a) implies $D^Cf_k=(D^Cg_{k, l+1}) \circ (D^Cg_{l, k+1})$ and $D^Cf_l=(D^Cg_{l, k+1}) \circ (D^Cg_{k, l+1})$ are bijections. So, the map $D^Cg_{l, k+1}$ induces on the set of $G_k$'s purple edges is a bijection. It is only left to show that two purple edges share a vertex in $G_k$ if and only if their $D^Cg_{l, k+1}$ images share a vertex in $G_l$.

If $[x, d_1]$ and $[x, d_2]$ are in $\mathcal{PI}(G_k)$, $D^Cg_{l, k+1}([x, d_1])=[Dg_{l, k+1}(x), Dg_{l, k+1}(d_1)]$ and $D^tg_{l, k+1}([x, d_2])= [Dg_{l, k+1}(x), Dg_{l, k+1}(d_2)]$ share $Dg_{l, k+1}(x)$. On the other hand, if $[w, d_3]$ and $[w, d_4]$ in $\mathcal{PI}(G_l)$ share $w$, then $[D^tg_{k, l+1}(\{w, d_3 \})]=[Dg_{k, l+1}(w), Dg_{k, l+1}(d_3)]$ and $[D^tg_{k, l+1}(\{w, d_4 \})]=[Dg_{k, l+1}(w), Dg_{k, l+1}(d_4)]$ share $Dg_{k, l+1}(w)$. Since $D^Cf_l$ is an isomorphism on $\mathcal{PI}(G_l)$, $D^Cg_{l, k+1}$ and $D^Cg_{k, l+1}$ act as inverses. So the preimages of $[w, d_3]$ and $[w, d_4]$ under $D^Cg_{l, k+1}$ share a vertex in $G_l$.\qedhere
\end{proof}

\vskip10pt

\noindent Lemmas \ref{L:irred} gives properties stemming from irreducibility (though not proving irreducibility):

\begin{lem}{\label{L:irred}} For each $1 \leq j \leq r$
~\\
\vspace{-6mm}
\begin{description}
\item[a.] there exists a $k$ such that either $e^u_k=E_{k,j}$ or $e^u_k= \overline{E_{k,j}}$ and \\[-5.5mm]
\item[b.] there exists a $k$ such that either $e^a_k=E_{k,j}$ or $e^a_k= \overline{E_{k,j}}$. \\[-5mm]
\end{description}
\end{lem}

\begin{proof} We start with (a). For contradiction's sake suppose there is some $j$ so that $e^u_k \neq E_{k,j}^{\pm 1}$ for all $k$. We inductively show $g(E_{0,j})=E_{0,j}$, implying $g$'s reducibility. Induction will be on the $k$ in $g_{k-1,1}$.

For the base case, we need $g_1(E_{0,j})=E_{1,j}$ if $e_1^u \neq E_{1,j}^{\pm 1}$.  $g_1$ is defined by $e^{pu}_0 \mapsto e^a_1 e^u_1$. Since $e_1^u \neq E_{1,j}$ and $\overline{e_1^u} \neq \overline{E_{(1,j)}}$, we know $e^{pu}_0 \neq E_{(0,j)}^{\pm 1}$. Thus, $g_1(E_{0,j})=E_{(1,j)}$, as desired.  Now inductively suppose $g_{k-1,1}(E_{0,j})=E_{k-1,j}$ and $e_k^u \neq E_{k,j}^{\pm 1}$. Then $e^{pu}_{k-1} \neq E_{k-1,j}^{\pm 1}$. Thus, since $e^{pu}_{k-1} \mapsto e^a_k e^u_k$ defines $g_k$, we know $g_k(E_{k-1,j})=E_{k,j}$. So $g_{k, 1}(E_{0,j})=g_k (g_{k-1,1}(E_{0,j}))=g_k (E_{k-1,j})= E_{k,j}$. Inductively, this proves $g(E_{0,j})=E_{0,j}$, we have our contradiction, and (b) is proved.

We now prove (b). For contradiction's sake, suppose that, for some $1 \leq j \leq r$, $e^a_k \neq E_{k,j}$ and $e^a_k \neq \overline{E_{k,j}}$ for each $k$. The goal will be to inductively show that, for each $E_{0,i}$ with $E_{0,i} \neq E_{0,j}$ and $E_{0,i} \neq \overline{E_{(0,j)}}$, $g(E_{0,i})$ does not contain $E_{0,j}$ and does not contain $\overline{E_{0,j}}$ (contradicting irreducibility).

We prove the base case. $g_1$ is defined by $e^{pu}_0 \mapsto e^a_1 e^u_1$. First suppose $E_{0,j}=(e^{pu}_0)^{\pm 1}$. Then $e^{pu}_0 \neq E_{0,i}^{\pm 1}$ (since $E_{0,i} \neq E_{0,j}^{\pm 1}$). So $g_1(E_{0,i})= E_{1,i}$, which does not contain $E_{1,j}^{\pm 1}$.  Now suppose that $E_{0,j} \neq e^{pu}_0$ and $E_{0,j} \neq \overline{e^{pu}_0}$. Then $e^a_1 e^u_1$ does not contain $E_{1,j}$ or $\overline{E_{1,j}}$ (since $e^a_k \neq (E_{k,j})^{\pm 1}$ by assumption). So $E_{1,j}^{\pm 1}$ are not in the image of $E_{0,i}$ if $E_{0,i} = e^{pu}_0$ (since the image of $E_{0,i}$ is then $e^a_1 e^u_1$) and are not in the image of $\overline {E_{0,i}}$ (since the image is $\overline{e^u_1} \overline{e^a_1}$) and are not in the image $E_{0,i}$ if $E_{0,i} \neq (e^{pu}_0)^{\pm 1}$ (since the image is $E_{1,i}$ and $E_{1,i} \neq E_{1,j}^{\pm 1}$). The base case is proved.

Inductively suppose $g_{k-1,1}(E_{0,i})$ does not contain $E_{k-1,j}^{\pm 1}$. Similar analysis as above shows $g_k(E_{k-1,i})$ does not contain $E_{k,j}^{\pm 1}$ for any $E_{k,i} \neq E_{k,j}^{\pm 1}$. Since $g_{k-1,1}(E_{k-1,i})$ does not contain $E_{k-1,j}^{\pm 1}$, $g_{k-1, 1}(E_{0,i})= e_1 \dots e_m$ with each $e_i \neq E_{k-1,j}^{\pm 1}$. Thus, no $g_k(e_i)$ contains $E_{k,j}^{\pm 1}$. So $g_{k, 1}(E_{0,i})= g_k(g_{k-1,1}(E_{0,i}))= g_k(e_1) \dots g_k(e_m)$ does not contain $E_{k,j}^{\pm 1}$. This completes the inductive step, thus (b). \qedhere
\end{proof}

\noindent \begin{rk} Lemma \ref{L:irred} is necessary, but not sufficient, for $g$ to be irreducible. For example, the composition of $a \mapsto ab$, $b \mapsto ba$, $c \mapsto cd$, and $d \mapsto dc$ satisfies Lemma \ref{L:irred}, but is reducible.
\end{rk}

\vskip8pt

\begin{proof}[Proof of Proposition \ref{P:am}]
AM property I follows from Proposition \ref{P:BC} and Lemma \ref{L:fk};
AM property II from Lemma \ref{L:IllegalTurn};
AM property III from Corollary \ref{C:UnachievedDirection};
AM property IV from Lemma \ref{L:EdgeImage};
AM property V from Lemma \ref{L:1Edge} and Corollary \ref{C:UnachievedDirection};
AM property VI from Lemma \ref{L:NG};
AM property VII from Lemma \ref{L:PurpleEdgeImages}; and
AM property VIII from Lemma \ref{L:irred}.
\qedhere
\end{proof}

\vskip10pt

\section{Lamination train track (ltt) structures}{\label{Ch:ltt}}

In Subection \ref{SS:Realltts} we defined ltt structures for ideally decomposed representatives with $(r;(\frac{3}{2}-r))$ potential. Both for defining $\mathcal{ID}$ diagrams and for applying the Birecurrency Condition, we need abstract definitions of ltt structures motivated by the AM properties of Section \ref{Ch:AMProperties}.

\subsection{Abstract lamination train track structures}

\begin{df} (See Example \ref{Ex:G(g)}) A \emph{lamination train track (ltt) structure $G$} is a pair-labeled colored train track graph (black edges will be included, but not considered colored) satisfying:
~\\
\vspace{-6mm}
\begin{description}
\item[ltt1:] Vertices are either purple or red. \\[-6mm]
\item[ltt2:] Edges are of 3 types ($\mathcal{E}_b$ comprises the black edges and $\mathcal{E}_c$ comprises the red and purple edges): \\[-6mm]
~\\
\vspace{-\baselineskip}
\indent \indent {\begin{description}
\item[(Black Edges):] A single black edge connects each pair of (edge-pair)-labeled vertices.  There are no other black edges. In particular, each vertex is contained in a unique black edge. \\[-5.7mm]
\item[(Red Edges):] A colored edge is red if and only if at least one of its endpoint vertices is red. \\[-5.7mm]
\item[(Purple Edges):] A colored edge is purple if and only if both endpoint vertices are purple. \\[-6mm]
    \end{description}}
\item[ltt3:] No pair of vertices is connected by two distinct colored edges. \\[-6mm]
    \end{description}

The purple subgraph of $G$ will be called the \emph{potential ideal Whitehead graph associated to $G$}, denoted \emph{$\mathcal{PI}(G)$}.  For a finite graph $\mathcal{G} \cong \mathcal{PI}(G)$, we say $G$ \emph{is an ltt Structure for $\mathcal{G}$}.

An \emph{$(r;(\frac{3}{2}-r))$ ltt structure} is an ltt structure $G$ for a $\mathcal{G} \in \mathcal{PI}_{(r;(\frac{3}{2}-r))}$ such that:
~\\
\vspace{-\baselineskip}
\begin{description}
\item[ltt(*)4:] $G$ has precisely 2r-1 purple vertices, a unique red vertex, and a unique red edge. \\[-6mm]
\end{description}

ltt structures are \emph{equivalent} that differ by an ornamentation-preserving (label and color preserving), homeomorphism.
\end{df}

\begin{nt}{\label{N:ltt}} \textbf{(ltt Structures)} For an ltt Structure $G$:
\begin{enumerate}[itemsep=-1.5pt,parsep=3pt,topsep=3pt]
\item An edge connecting a vertex pair $\{d_i, d_j \}$ will be denoted [$d_i, d_j$], with interior ($d_i, d_j$).  \newline
\indent (While the notation [$d_i, d_j$] may be ambiguous when there is more than one edge connecting the vertex pair $\{d_i, d_j \}$, we will be clear in such cases as to which edge we refer to.)
\item $[e_i]$ will denote [$d_i, \overline{d_i}$]
\item Red vertices and edges will be called \emph{nonperiodic}.
\item Purple vertices and edges will be called \emph{periodic}.
\item \emph{$\mathcal{C}(G)$} will denote the colored subgraph of $G$, called the \emph{colored subgraph associated to} (or \emph{of}) $G$.
\item $G$ will be called \emph{admissible} if it is birecurrent.
\end{enumerate}
\noindent For an $(r;(\frac{3}{2}-r))$ ltt structure $G$ for $\mathcal{G}$, additionally:
\begin{enumerate}[itemsep=-1.5pt,parsep=3pt,topsep=3pt]
\item $d^u$ will label the unique red vertex and be called the \emph{unachieved direction}.
\item $e^R=[t^R]$, will denote the unique red edge and $\overline{d^a}$ its purple vertex's label. So $t^R= \{d^u, \overline{d^a} \}$ and $e^R= [d^u, \overline{d^a}]$.
\item $\overline{d^a}$ is contained in a unique black edge, which we call the \emph{twice-achieved edge}.
\item $d^a$ will label the other twice-achieved edge vertex and be called the \emph{twice-achieved direction}.
\item If $G$ has a subscript, the subscript carries over to all relevant notation. For example, in $G_k$, $d^u_k$ will label the red vertex and $e^R_k$ the red edge.
\end{enumerate}
\end{nt}

A 2r-element set of the form $\{x_1, \overline{x_1}, \dots, x_r, \overline{x_r} \}$, with elements paired into \emph{edge pairs} $\{x_i, \overline{x_i}\}$, will be called a \emph{rank}-$r$ \emph{edge pair labeling set}. It will then be standard to say $\overline{\overline{x_i}}=x_i$. A graph with vertices labeled by an edge pair labeling set will be called a \emph{pair-labeled} graph. If an indexing is prescribed, it will be called an \emph{indexed pair-labeled} graph.

\begin{df}{\label{D:Indexedltt*}}
For an ltt structure to be considered \emph{indexed pair-labeled}, we require:
\begin{enumerate}[itemsep=-1.5pt,parsep=3pt,topsep=3pt]
\item It is index pair-labeled (of rank $r$) as a graph. \\[-6mm]
\item The vertices of the black edges are indexed by edge pairs. \\[-6mm]
\end{enumerate}
Index pair-labeled ltt structures are \emph{equivalent} that are equivalent as ltt structures via an equivalence preserves the indexing of the vertex labeling set.

By index pair-labeling (with rank $r$) an $(r;(\frac{3}{2}-r))$ ltt structure $G$ and edge-indexing the edges of an $r$-petaled rose $\Gamma$, one creates an identification of the vertices in $G$ with $\mathcal{D}(v)$, where $v$ is the vertex of $\Gamma$.  With this identification, we say $G$ is \emph{based} at $\Gamma$. In such a case it will be standard to use the notation $\{d_1, d_2, \dots, d_{2r-1}, d_{2r} \}$ for the vertex labels (instead of $\{x_1, x_2, \dots, x_{2r-1}, x_{2r} \}$). Additionally, $[e_i]$ will denote $[D_0(e_i), D_0(\overline{e_i})] = [d_i, \overline{d_i}]$ for each edge $e_i \in \mathcal{E}(\Gamma)$.

A $\mathcal{G} \in \mathcal{PI}_{(r;(\frac{3}{2}-r))}$ will be called \emph{(index) pair-labeled} if its vertices are labeled by a $2r-1$ element subset of the rank $r$ (indexed) edge pair labeling set.
\end{df}

\subsection{Maps of lamination train track structures}{\label{SS:BasedlttStructureMaps}}

\emph{\textbf{Let $G$ and $G'$ be rank-$r$ indexed pair-labeled $(r;(\frac{3}{2}-r))$ ltt structures, with bases $\Gamma$ and $\Gamma'$, and $g:\Gamma \to \Gamma'$ a tight homotopy equivalence taking edges to nondegenerate edge-paths.}}

Recall that $Dg$ induces a map of turns $D^tg: \{a,b\} \mapsto \{Dg(a), Dg(b)\}$. $Dg$ additionally induces a map on the corresponding edges of $\mathcal{C}(G)$ and $\mathcal{C}(G')$ if the appropriate edges exist in $\mathcal{C}(G')$:

\begin{df} When the map sending
\begin{enumerate}[itemsep=-1.5pt,parsep=3pt,topsep=3pt]
\item the vertex labeled $d$ in $G$ to that labeled by $Dg(d)$ in $G'$ and
\item the edge [$d_i, d_j$] in $\mathcal{C}(G)$ to the edge [$Dg(d_i), Dg(d_j)$] in $\mathcal{C}(G')$ \newline
\noindent also satisfies that
\item [3.] each $\mathcal{PI}(G)$ is mapped isomorphically onto $\mathcal{PI}(G')$,
\end{enumerate}
\noindent we call it the \emph{map of colored subgraphs induced by $g$} and denote it $D^C(g): C(G) \to C(G')$.

When it exists, the map $D^T(g): G \to G'$ \emph{induced by $g$} is the extension of $D^C(g): C(G) \to C(G')$ taking the interior of the black edge of $G$ corresponding to the edge $E \in \mathcal{E}(\Gamma)$ to the interior of the smooth path in $G'$ corresponding to $g(E)$.
\end{df}

\subsection{ltt structures are ltt structures}{\label{Ch:lttMeansltt}}

By showing that the ltt structures of Subsection \ref{SS:Realltts} are indeed abstract ltt structures, we can create a finite list of ltt structures for a particular $\mathcal{G} \in \mathcal{PI}_{(r;(\frac{3}{2}-r))}$ to apply the birecurrency condition to.

\begin{lem}{\label{L:PF}} Let $g:\Gamma \to \Gamma$ be a representative of $\phi \in Out(F_r)$, with $(r;(\frac{3}{2}-r))$ potential, such that $\mathcal{IW}(g) \cong \mathcal{G}$. Then $G(g)$ is an $(r;(\frac{3}{2}-r))$ ltt structure with base graph $\Gamma$. Furthermore, $\mathcal{PI}(G(g)) \cong \mathcal{G}$.
\end{lem}

\begin{proof}  This is more or less just direct applications of the lemmas above. \cite{p12a} gives a detailed proof of a more general lemma. \qedhere
\end{proof}

\subsection{Generating triples}{\label{SS:GeneratingTriples}}

\noindent Since we deal with representatives decomposed into Nielsen generators, we use an abstract notion of an ``indexed generating triple.''

\begin{df}{\label{D:Triple}}  A \emph{triple} $(g_k, G_{k-1}, G_k)$ will be an ordered set of three objects where $g_k: \Gamma_{k-1} \to \Gamma_k$ is a proper full fold of roses and, for $i=k-1,k$, $G_i$ is an ltt structure with base $\Gamma_i$.
\end{df}

\begin{df}{\label{D:GeneratingTriple}}  A \emph{generating triple} is a triple $(g_k, G_{k-1}, G_k)$ where
~\\
\vspace{-\baselineskip}
\begin{description}
\item [(gtI)] $g_k: \Gamma_{k-1} \to \Gamma_k$ is a proper full fold of edge-indexed roses defined by
\begin{itemize}
\item [a.] $g_k(e_{k-1,j_k})= e_{k,i_k} e_{k,j_k}$  where $d^a_k=D_0(e_{k,i_k})$, $d^u_k=D_0(e_{k,j_k})$, and $e_{k,i_k} \neq (e_{k,j_k})^{\pm 1}$ and \\[-5mm]
\item [b.] $g_k(e_{k-1,t})= e_{k,t}$ for all $e_{k-1,t} \neq (e_{k,j_k})^{\pm 1}$;
\end{itemize}
\item [(gtII)] $G_i$ is an indexed pair-labeled $(r;(\frac{3}{2}-r))$ ltt structure with base $\Gamma_i$ for $i=k-1,k$; and
\item [(gtIII)] The induced map of based ltt structures $D^T(g_k): G_{k-1} \to G_k$ exists and, in particular, restricts to an isomorphism from $\mathcal{PI}(G_{k-1})$ to $\mathcal{PI}(G_k)$.
\end{description}
\end{df}

\vskip10pt

\begin{nt}{\label{N:GeneratingTriples}} \textbf{(Generating Triples)} For a generating triple $(g_k, G_{k-1}, G_k)$:
\begin{enumerate}[itemsep=-1.5pt,parsep=3pt,topsep=3pt]
\item We call $G_{k-1}$ the \emph{source ltt structure} and $G_k$ the \emph{destination ltt structure}. \\[-5mm]
\item{\label{N:IngoingGeneratorTerminology}}  $g_k$ will be called the \emph{(ingoing) generator} and will sometimes be written $g_k: e^{pu}_{k-1} \mapsto e^a_k e^u_k$ (``p'' is for ``pre'').  Thus, $d_{k-1,j_k}$ will sometimes be written $d^{pu}_{k-1}$. \\[-5mm]
\item $e^{pa}_{k-1}$ denotes $e_{k-1,i_k}$ (again ``p'' is for ``pre''). \\[-5mm]
\item If $G_k$ and $G_{k-1}$ are indexed pair-labeled $(r;(\frac{3}{2}-r))$ ltt structures for $\mathcal{G}$, then $(g_k, G_{k-1}, G_k)$ will be a generating triple \emph{for $\mathcal{G}$}.
\end{enumerate}
\end{nt}

\begin{rk}
While $d^u_i$ is determined by the red vertex of $G_i$ (and does not rely on other information in the triple), $d^{pu}_{k-1}$ and $d^{pa}_{k-1}$ actually rely on gtI, and cannot be determined by knowing only $G_{k-1}$.
\end{rk}

\begin{ex}{\label{E:GeneratingTripleEx}} The triple $(g_2, G_1, G_2)$ of Example \ref{Ex:InducedMap} is an example of a generating triple where $x$ denotes both $E_{(1,1)}$ and $E_{(2,1)}$, $y$ denotes both $E_{(1,2)}$ and $E_{(2,2)}$, and $z$ denotes both $E_{(1,3)}$ and $E_{(2,3)}$.
\end{ex}

\begin{df}{\label{D:GeneratorExtendsTolttStructures}}
Suppose $(g_i, G_{i-1}, G_i)$ and $(g_i', G_{i-1}', G_i)'$ are generating triples. Let $g_i^T: G_{i-1} \to G_i$ be induced by  $g_i: \Gamma_{i-1} \to \Gamma_i$ and $g_i^T: G_{i-1}' \to G_i'$ by $g_i: \Gamma_{i-1}' \to \Gamma_i'$. We say $(g_i, G_{i-1}, G_i)$ and $(g_i', G_{i-1}', G_i')$ are \emph{equivalent} if there exist indexed pair-labeled graph equivalences $H_{i-1}: \Gamma_{i-1} \to \Gamma_{i-1}'$ and $H_i: \Gamma_i \to \Gamma_i'$ such that:
\begin{enumerate}[itemsep=-1.5pt,parsep=3pt,topsep=3pt]
\item for $k=i,i-1$, $H_i: \Gamma_i \to \Gamma_i'$ induces indexed pair-labeled ltt structure equivalence of $G_i$ and $G_i'$
\item and $H_i \circ g_i = g_i' \circ H_{i-1}$.
\end{enumerate}
\end{df}

\section{Peels, extensions, and switches}{\label{Ch:Peels}}

\noindent Suppose $\mathcal{G} \in \mathcal{PI}_{(r;(\frac{3}{2}-r))}$. By Section \ref{Ch:IdealDecompositions}, if there is a $\phi \in \mathcal{AFI}_r$ with $IW(\phi) \cong \mathcal{G}$, then there is an ideally decomposed $(r;(\frac{3}{2}-r))$-potential representative $g$ of a power of $\phi$. By Section \ref{Ch:AMProperties}, such a representative would satisfy the AM properties. Thus, if we can show that a representative satisfying the properties does not exist, we have shown there is no $\phi \in \mathcal{AFI}_r$ with $IW(\phi) \cong \mathcal{G}$ (we use this fact in Section \ref{Ch:UnachievableGraphs}). In this section we show what triples $(g_k, G_{k-1}, G_k)$ satisfying the AM properties must look like. We prove in Proposition \ref{P:ExtensionsSwitches} that, if the structure $G_k$ and a purple edge $[d, d^a_k]$ in $G_k$ are set, then there is only one $g_k$ possibility and at most two $G_{k-1}$ possibilities (one generating triple possibility will be called a ``switch'' and the other an ``extension''). Extensions and switches are used here only to define ideal decomposition diagrams but have interesting properties used (and proved) in \cite{p12c} and \cite{p12d}.

\subsection{Peels}{\label{S:Peels}}

\noindent As a warm-up, we describe a geometric method for visualizing ``switches'' and ``extensions'' as moves, ``peels,'' transforming an ltt structure $G_i$ into an ltt structure $G_{i-1}$.

Each peel of an ltt structure $G_i$ involves three directed edges of $G_i$:
~\\
\vspace{-\baselineskip}
\begin{itemize}
\item The \emph{First Edge of the Peel} (\emph{New Red Edge} in $G_i$): the red edge from $d^u_i$ to $\overline{d^a_i}$. \\[-6mm]
\item The \emph{Second Edge of the Peel} (\emph{Twice-Achieved Edge} in $G_i$): the black edge from $\overline{d^a_i}$ to $d^a_i$. \\[-6mm]
\item The \emph{Third Edge of the Peel} (\emph{Determining Edge} for the peel): a purple edge from $d^a_i$ to $d$.  (In $G_{i-1}$, this vertex $d$ will be the red edge's attaching vertex, labeled $\overline{d^a_{i-1}}$). \\[-6mm]
\end{itemize}
\begin{figure}[H]
\centering
\includegraphics[width=2.8in]{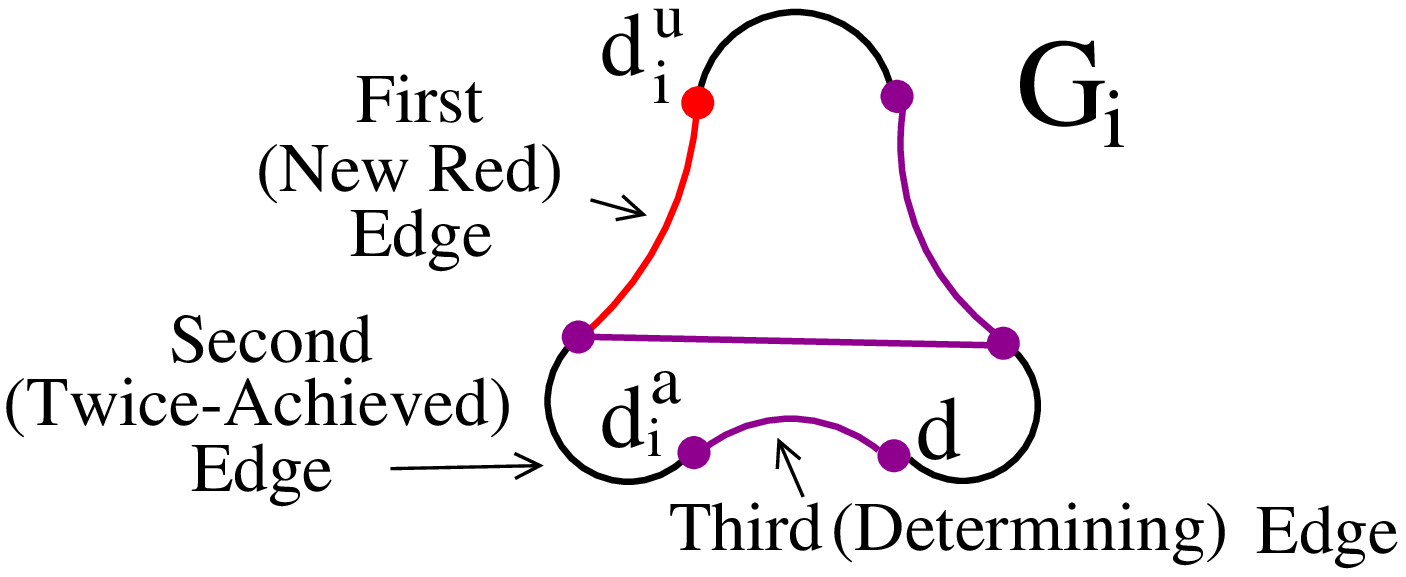}
\end{figure}

For each determining edge choice $[d^a_i, d]$ in $G_i$, there is one ``peel switch'' (Figure 8) and one ``peel extension'' (Figure 7). When $G_i$ has only a single purple edge at $d^a_i$, the switch and extension differ by a color switch of two edges and two vertices. We start by explaining this case. After, we explain the preliminary step necessary for any switch where more than one purple edge in $G_i$ contains $d^a_i$.

We describe how, when $G_i$ has only a single purple edge at $d^a_i$, the two peels determined by $[d^a_i, d]$ transform $G_i$ into $G_{i-1}$. While keeping $d$ fixed, starting at vertex $\overline{d^a_i}$, peel off black edge $[\overline{d^a_i}, d^a_i]$ and the third edge $[d^a_i, d]$, leaving copies of $[\overline{d^a_i}, d^a_i]$ and $[d^a_i, d]$ and creating a new edge $[d^u_i, d]$ from the concatenation of the peel's first, second, and third edges (Figure 7 or 8).

In a \emph{peel extension}: $[d^u_i, \overline{d^a_i}]$ disappears into the concatenation and does not exist in $G_{i-1}$, the copy of $[\overline{d^a_i}, d^a_i]$ left behind stays black in $G_{i-1}$, the copy of $[d^a_i, d]$ left behind stays purple in $G_{i-1}$, the edge $[d^u_i, d]$ formed from the concatenation is red in $G_{i-1}$, and nothing else changes from $G_i$ to $G_{i-1}$ (if one ignores the first indices of the vertex labels). The triple $(g_i, G_{i-1}, G_i)$, with $g_i$ as in AM property VI, will be called the \emph{extension determined by $[d^a_i, d]$}.
~\\
\vspace{-6mm}
\begin{figure}[H]{\label{fig:PeelExtension}}
\centering
\includegraphics[width=2.6in]{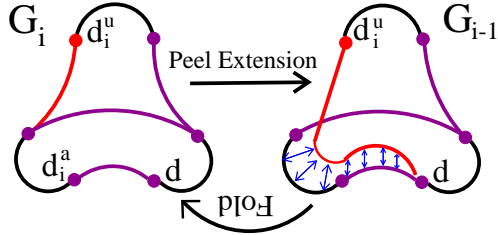}
\hspace{1.8in}\parbox{6in}{\caption{{\small{\emph{Peel Extension: Note that the first, second, and third edges of the peel concatenate to form the red edge $[d^u_i, d]$ in $G_{i-1}$ and that copies of $[\overline{d^a_i}, d^a_i]$ and $[d^a_i, d]$ remain in $G_{i-1}$.}}}}} \\[-1.5mm]
\end{figure}

\indent In a \emph{peel switch} (where $[d^a_i, d]$ was the only purple edge in $G_i$ containing $d^a_i$): Again $[d^u_i, \overline{d^a_i}]$ has disappeared into the concatenation and the copy of $[\overline{d^a_i}, d^a_i]$ left behind stays black in $G_{i-1}$. But now the edge $[d^u_i, d]$ formed from the concatenation is purple in $G_{i-1}$, the copy of $[d^a_i, d]$ left behind and vertex $d^a_i$ are both red in $G_{i-1}$ (so that $d^a_i$ is now actually $d^u_{i-1}$), and vertex $d^u_i$ is purple in $G_{i-1}$. The triple $(g_i, G_{i-1}, G_i)$, with $g_i$ as in AM property VI, will be called the \emph{switch determined by $[d^a_i, d]$}.
~\\
\vspace{-6mm}
\begin{figure}[H]{\label{fig:PeelSwitch}}
\centering
\includegraphics[width=2.4in]{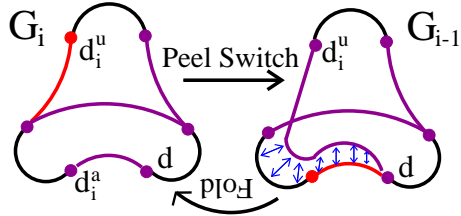}\
\hspace{1.8in}\parbox{5.5in}{\caption{{\small{\emph{Peel Switch (when $d^a_i$ only belongs to one purple edge in $G_{i-1}$): The first, second, and third edges of the peel concatenate to form a purple edge $[d^u_i, d]$ in $G_{i-1}$.  The determining edge $[d^a_i, d]$ is the red edge of $G_{i-1}$, with red vertex $d^a_i$.}}}}} \\[-1.5mm]
\end{figure}

\indent Preliminary step for a switch where purple edges other than the determining edge $[d^a_i, d]$ contain vertex $d^a_i$ in $G_i$: For each purple edge $[d^a_i, d']$ in $G_i$ where $d \neq d'$, form a purple concatenated edge $[d', d^u_i]$ in $G_{i-1}$ by concatenating $[d', d^a_i]$ with a copy of $[d^a_i, \overline{d^a_i}, d^u_i]$, created by splitting open, as in Figure 9, $[d^a_i, \overline{d^a_i}]$ from $d^a_i$ to $\overline{d^a_i}$ and $[\overline{d^a_i}, d^u_i]$ from $\overline{d^a_i}$ to $d^u_i$.
~\\
\vspace{-6mm}
\begin{figure}[H]
\centering
\includegraphics[width=4.1in]{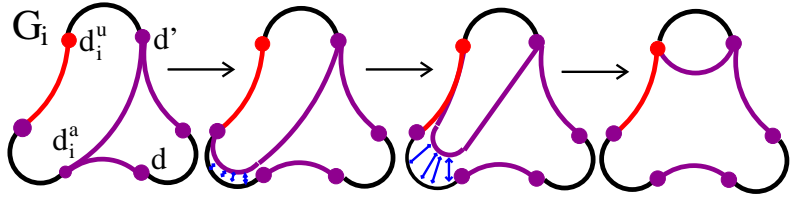}
\hspace{1.5in}\parbox{6in}{\caption{\small{\emph{Peel Switch Preliminary Step: For each purple edge $[d^a_i, d']$ in $G_i$, the peeler peels a copy of $[d^a_i, \overline{d^a_i}, d^u_i]$ off to concatenate with $[d^a_i, d']$ and form the purple edge $[d^u_i, d']$.}}}}
\label{fig:PeelSwitch2} \\[-1.5mm]
\end{figure}

\indent To check the peel switch was performed correctly, one can: remove $G_i$'s red edge, lift vertex $d^a_i$ (with purple edges containing it dangling from one's fingers), and drop vertex $d^a_i$ in the spot of vertex $d^u_i$, while leaving behind a copy of $[d^a_i, d]$ to become the new red edge of $G_{i-1}$ (with $d^{pa}_{i-1}$ as the red vertex).

\subsection{Extensions and switches}

\emph{\textbf{Throughout this section $G_k$ will be an indexed pair-labeled $(r;(\frac{3}{2}-r))$ ltt structure for a $\mathcal{G} \in \mathcal{PI}_{(r;(\frac{3}{2}-r))}$ with rose base graph $\Gamma_k$. We use the standard notation.}}

We define extensions and switches ``entering'' an indexed pair-labeled admissible $(r;(\frac{3}{2}-r))$ ltt structure $G_k$ for $\mathcal{G}$. However, we first prove that determining edges exist.

\vskip10pt

\begin{lem} There exists a purple edge with vertex $d^a_k$, so that it may be written $[d^a_k, d_{k,l}]$.
\end{lem}

\begin{proof} If $d^a_k$ were red, the $e^R_k$ would be $[d^a_k, \overline{d^a_k}]$, violating that $\mathcal{G} \in \mathcal{PI}_{(r;(\frac{3}{2}-r))}$. $d^a_k$ must be contained in an edge $[d^a_k, d_{k,l}]$ or $\mathcal{G}$ would not have 2r-1 vertices. If $d_{k,l}$ were red, i.e. $d_{k,l}=d^u_k$, then both $[d^u_k, \overline{d^a_k}]$ and $[d^u_k, d^a_k]$ would be red, violating [ltt(*)4]. So $[d^a_k, d_{k,l}]$ must be purple.  \qedhere
\end{proof}

\begin{df}{\label{D:Extension}} (See Figure  \ref{fig:ExtensionDiagram}) For a purple edge $[d^a_k, d_{k,l}]$ in $G_k$, the \emph{extension determined by} $[d^a_k, d_{k,l}]$, is the generating triple $(g_k, G_{k-1}, G_k)$ for $\mathcal{G}$ satisfying:
~\\
\vspace{-\baselineskip}
\begin{description}
\item[(extI):] The restriction of $D^T(g_k)$ to $\mathcal{PI}(G_{k-1})$ is defined by sending, for each $j$, the vertex labeled $d_{k-1,j}$ to the vertex labeled $d_{k,j}$ and extending linearly over edges. \\[-5mm]
\item[(extII):] $d^u_{k-1}= d^{pu}_{k-1}$, i.e. $d^{pu}_{k-1}= d_{k-1,j_k}$ labels the single red vertex in $G_{k-1}$. \\[-5mm]
\item[(extIII):] $\overline{d^a_{k-1}}= d_{k-1,l}$.
\end{description}
\end{df}

\begin{rk}
(extIII) implies that the single red edge $e^{R}_{k-1}= [d^u_{k-1}, \overline{d^{a}_{k-1}}]$ of $G_{k-1}$ can be written, among other ways, as $[d^{pu}_{k-1}, d_{(k-1,l)}]$.
\end{rk}

\vskip10pt

\noindent Explained in Section \ref{S:Peels}, but with this section's notation, an extension transforms ltt structures as:
~\\
\vspace{-6.5mm}
\begin{figure}[H]
\centering
\includegraphics[width=4.3in]{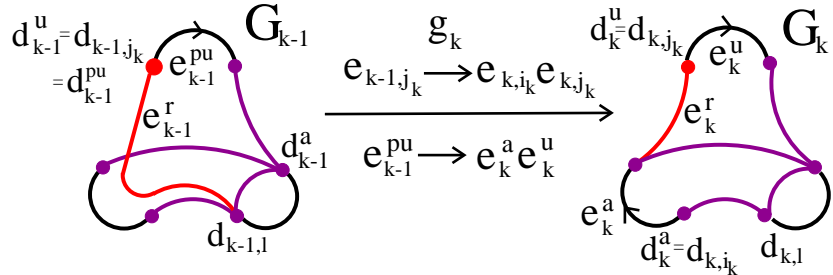}
~\\
\vspace{-5mm}
\caption[Extension]{{\small{\emph{Extension}}}}
\label{fig:ExtensionDiagram}
\end{figure}

\begin{lem}{\label{L:ExtensionUniqueness}} Given an edge $[d^a_k, d_{k,l}]$ in $\mathcal{PI}(G_k)$, the extension $(g_k, G_{k-1}, G_k)$ determined by $[d^a_k, d_{k,l}]$ is unique.
~\\
\vspace{-6.5mm}
\begin{description}
\item [I.] $G_{k-1}$ can be obtained from $G_k$ by the following steps:
~\\
\vspace{-6.5mm}
{\begin{enumerate}[itemsep=-1.5pt,parsep=3pt,topsep=3pt]
\item removing the interior of the red edge from $G_k$; \\[-5.5mm]
\item replacing each vertex label $d_{k,i}$ with $d_{k-1,i}$ and each vertex label $\overline{d_{k,i}}$ with $\overline{d_{k-1,i}}$; and \\[-5.5mm]
\item adding a red edge $e^R_{k-1}$ connecting the red vertex to $d_{k-1,l}$. \\[-5.5mm]
\end{enumerate}}
\item [II.] The fold is such that the corresponding homotopy equivalence maps the oriented $e_{k-1,j_k} \in \mathcal{E}_{k-1}$ over the path $e_{k,i_k} e_{k,j_k}$ in $\Gamma_k$ and then each oriented $e_{k-1,t} \in \mathcal{E}_{k-1}$ with $e_{k-1,t} \neq e_{k-1,j_k}^{\pm 1}$ over $e_{k,t}$.
\end{description}
\end{lem}

\begin{proof}  The proof is an unraveling of definitions. A full presentation can be found in \cite{p12a}.  \qedhere
\end{proof}

\vskip15pt

\begin{df}{\label{D:Switch}} (See Figure \ref{fig:SwitchDiagram}) The \emph{switch} determined by a purple edge $[d^a_k, d_{(k,l)}]$ in $G_k$ is the generating triple $(g_k, G_{k-1}, G_k)$ for $\mathcal{G}$ satisfying:
~\\
\vspace{-\baselineskip}
\indent \begin{description}
\item[(swI):] $D^T(g_k)$ restricts to an isomorphism from $\mathcal{PI}(G_{k-1})$ to $\mathcal{PI}(G_k)$ defined by
    $$\mathcal{PI}(G_{k-1}) \xrightarrow{d^{pu}_{k-1} \mapsto d^a_k=d_{k, i_k}} \mathcal{PI}(G_k)$$
    ($d_{k-1,t} \mapsto d_{k,t}$ for $d_{k-1,t} \neq d^{pu}_{k-1}$) and extended linearly over edges. \\[-5mm]
\item[(swII):] $d^{pa}_{k-1} = d^u_{k-1}$. \\[-5mm]
\item[(swIII):] $\overline{d^a_{k-1}} = d_{k-1,l}$.
\end{description}
\end{df}

\vskip10pt

\begin{rk} (swII) implies that the red edge $e^R_{k-1} = [d^u_{k-1}, d^a_{k-1}]$ of $G_{k-1}$ can be written $[d^{pa}_{k-1}, \overline{d^a_{k-1}}]$, among other ways. (swIII) implies that $e^R_{k-1}$ can be written $[d_{(k-1,i_k)}, d_{(k-1,l)}]$.
\end{rk}

\vskip10pt

\noindent Explained in Section \ref{S:Peels}, but with this section's notation, a switch transforms ltt structures as follows:
~\\
\vspace{-7mm}
\begin{figure}[H]
\centering
\includegraphics[width=4.3in]{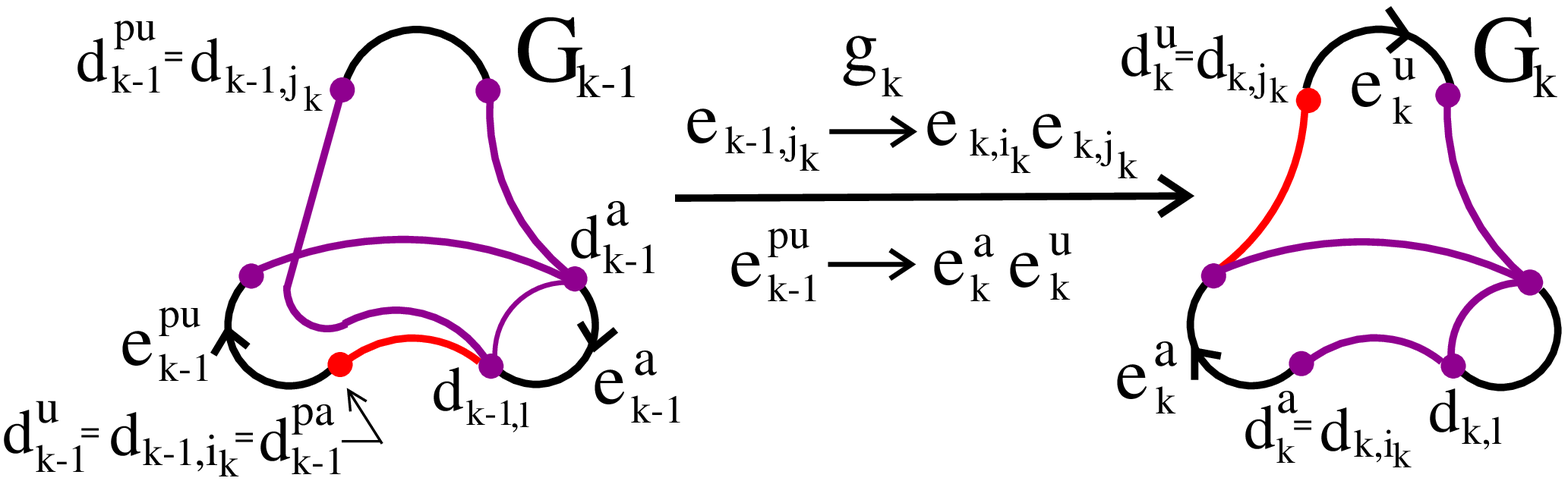}
~\\
\vspace{-5.5mm}
\caption[Switch]{{\small{Switch}}}
\label{fig:SwitchDiagram}
\end{figure}

\begin{lem}{\label{L:SwitchUniqueness}} Given an edge $[d^a_k, d_{k,l}]$ in $\mathcal{PI}(G_k)$, the switch $(g_k, G_{k-1}, G_k)$ determined by $[d^a_k, d_{k,l}]$ is unique.
~\\
\vspace{-\baselineskip}
\begin{description}
\item [I.] $G_{k-1}$ can be obtained from $G_k$ by the following steps:
~\\
\vspace{-6.5mm}
{\begin{enumerate}[itemsep=-1.5pt,parsep=3pt,topsep=3pt]
\item Start with $\mathcal{PI}(G_k)$. \\[-5.4mm]
\item Replace each vertex label $d_{k,i}$ with $d_{k-1,i}$. \\[-5.4mm]
\item Switch the attaching (purple) vertex of the red edge to be $d_{k-1,l}$. \\[-5.4mm]
\item Switch the labels $d_{(k-1,j_k)}$ and $d_{(k-1,i_k)}$, so that the red vertex of $G_{k-1}$ will be $d_{k-1,i_k}$ and the red edge of $G_{k-1}$ will be $[d_{(k-1,i_k)}, d_{(k-1,l)}]$. \\[-5.4mm]
\item Include black edges connecting inverse pair labeled vertices (there is a black edge $[d_{(k-1,i)}, d_{(k-1,j)}]$ in $G_{k-1}$ if and only if there is a black edge $[d_{(k,i)}, d_{(k,j)}]$ in $G_k$). \\[-5.5mm]
\end{enumerate}}
\item [II.] The fold is such that the corresponding homotopy equivalence maps the oriented $e_{k-1,j_k} \in \mathcal{E}_{k-1}$ over the path $e_{k,i_k} e_{k,j_k}$ in $\Gamma_k$ and then each oriented $e_{k-1,t} \in \mathcal{E}_{k-1}$ with $e_{k-1,t} \neq e_{k-1,j_k}^{\pm 1}$ over $e_{k,t}$.
\end{description}
\end{lem}

\begin{proof} The proof is an unraveling of definitions. A full presentation can be found in \cite{p12a}. \qedhere
\end{proof}

\vskip10pt

Recall (Proposition \ref{P:am}) that each triple in an ideal decomposition satisfies the AM properties. Thus, to construct a diagram realizing any ideally decomposed $(r;(\frac{3}{2}-r))$-potential representative with ideal Whitehead graph $\mathcal{G}$, we want edges of the diagram to correspond to triples satisfying the AM properties. Proposition \ref{P:ExtensionsSwitches} tells us each such a triple is either an admissible switch or admissible extension.

\begin{prop}{\label{P:ExtensionsSwitches}} Suppose $(g_k, G_{k-1}, G_k)$ is a triple for $\mathcal{G}$ such that: \newline
\indent 1. $\mathcal{G} \in \mathcal{PI}_{(r;(\frac{3}{2}-r))}$ and \newline
\indent 2. $G_i$ is an indexed pair-labeled $(r;(\frac{3}{2}-r))$ ltt structure for $\mathcal{G}$ with base graph $\Gamma_i$, for $i=k,k-1$. \newline
\noindent Then $(g_k, G_{k-1}, G_k)$ satisfies AM properties I-VII if and only if it is either an admissible switch or an admissible extension. \newline
\indent In particular, in the circumstance where $d^u_{k-1}= d^{pa}_{k-1}$, the triple is a switch and, in the circumstance where $d^u_{k-1}= d^{pu}_{k-1}$, the triple is an extension.
\end{prop}

\begin{proof} For the forward direction, assume $(g_k, G_{k-1}, G_k)$ satisfies AM properties I-VII and (1)-(2) in the proposition statement. We show the triple is either a switch or an extension (AM property I give birecurrency). Assumption (1) in the proposition statement implies (gtII).

By AM property VI, $g_k$ is defined by $g_k(e^{pu}_{k-1})=e^a_k e^u_k$ and $g_k(e_{k-1,i})=e_{k,i}$ for $e_{k-1,i} \neq (e^{pu}_{k-1})^{\pm 1}$, $D_0(e^u_k)= d^u_k$, $D_0(\overline{e^a_k})= \overline{d^a_k}$, and $e^{pu}_{k-1}= e_{(k-1,j)}$, where $e^u_k=e_{k,j}$. We have (gtI).

By AM property VII, $Dg_k$ induces on isomorphism from $SW(G_{k-1})$ to $SW(G_k)$. Since the only direction whose second index is not fixed by $Dg_k$ is $d^{pu}_{k-1}$, the only vertex label of $SW(G_{k-1})$ not determined by this isomorphism is the preimage of $d^a_k$ (which AM property IV dictates to be either $d^{pu}_{k-1}$ or $d^{pa}_{k-1}$).  When the preimage is $d^{pa}_{k-1}$, this gives (extI). When the preimage is $d^{pu}_{k-1}$, this gives (swI). For the isomorphism to extend linearly over edges, we need that images of edges in $G_{k-1}$ are edges in $G_k$, i.e. $[Dg_k(d_{k-1,i}), Dg_k(d_{k-1,j})]$ is an edge in $G_k$ for each edge $[d_{(k-1,i)}, d_{(k-1,j)}]$ in $G_{k-1}$. This follows from AM property IV. We have (gtIII).

AM property II gives either $d^u_{k-1}= d^{pa}_{k-1}$ or $d^u_{k-1}=d^{pu}_{k-1}$. In the switch case, the above arguments imply $d^{pu}_{k-1}$ labels a purple vertex. So $d^u_{k-1}=d^{pa}_{k-1}$ (since AM property III tells us $d^u_{k-1}$ is red). This gives (swII) once one appropriately coordinates notation. In the extension case, the above arguments give instead that $d^{pa}_{k-1}$ labels a purple vertex, meaning $d^u_{k-1}=d^{pu}_{k-1}$ (again since AM property III tells us $d^u_{k-1}$ is red). This gives us (extII). We are left with (extIII) and (swIII). What we need is that $[d^a_k, d_{k,l}]$ is a purple edge in $G_k$ where $\overline{d^a_{k-1}}= d_{k-1,l}$.

By AM property V, $G_{k-1}$ has a single red edge $[t^R_{k-1}] = [\overline{d^a_{k-1}}, d^u_{k-1}]$. By AM property IV, $D^Cg_k([t^R_{k-1}])$ is in $\mathcal{PI}(G_k)$. First consider what we established is the switch case, i.e. assume $d^u_{k-1}=d^{pa}_{k-1}$.  The goal is to determine $[t^R_{k-1}]= [d_{(k-1,i_k)}, d_{(k-1,l)}]$, where $d^a_k=d_{k,i_k}$ ($d_{k-1,i_k}=d^{pa}_{k-1}$) and $[d^a_k, d_{k,l}]$ is in $\mathcal{PI}(G_k)$ (making $(g_k, G_{k-1}, G_k)$ the switch determined by $[d^a_k, d_{k,l}]$). Since $d^u_{k-1}=d^{pa}_{k-1}$, we know $[t^R_{k-1}]= [\overline{d^a_{k-1}}, d^u_{k-1} ]=[\overline{d^a_{k-1}}, d^{pa}_{k-1}]$.  We know $\overline{d^a_{k-1}} \neq d^{pa}_{k-1}$ (since (tt2) implies $\overline{d^a_{k-1}} \neq d^u_{k-1}$, which equals $d^{pa}_{k-1}$). Thus, AM property VI says $D^Cg_k([t^R_{k-1}]) = D^Cg_k([\overline{d^a_{k-1}}, d^{pa}_{k-1}]) = [d_{(k,l)}, d^{a}_{k}]$ where $\overline{d^a_{k-1}}=e_{k-1,l}$. So $[d_{(k,l)}, d^{a}_{k}]$ is in $\mathcal{PI}(G_k)$. We thus have (swIII). Now consider what we established is the extension case, i.e. assume $d^u_{k-1}= d^{pu}_{k-1}$. We need $[t^R_{k-1}]=[d_{(k-1,j_k)}, d_{(k-1,l)}]$, where $d^u_{k-1}=d_{k-1,j_k}$ and $[d^a_k, d_{k,l}]$ is in $\mathcal{PI}(G_k)$ (making $(g_k, G_{k-1}, G_k)$ the extension determined by $[d^a_k, d_{k,l}]$). Since $d^u_{k-1}= d^{pu}_{k-1}$, we know $[t^R_{k-1}]= [\overline{d^a_{k-1}}, d^u_{k-1} ]= [\overline{d^a_{k-1}}, d^{pu}_{k-1}]$. We know $\overline{d^a_{k-1}} \neq d^{pu}_{k-1}$ (since (tt2) implies $\overline{d^a_{k-1}} \neq d^u_{k-1}$, which equals $d^{pu}_{k-1}$). Thus, by AM property VI, $D^Cg_k([t^R_{k-1}]) = D^Cg_k([\overline{d^a_{k-1}}, d^{pu}_{k-1}]) = [d_{(k,l)}, d^{a}_{k}]$, where $\overline{d^a_{k-1}}= e_{k-1,l}$. We have (extIII) and the forward direction.

For the converse, assume $(g_k, G_{k-1}, G_k)$ is either an admissible switch or extension. Since we required extensions and switches be admissible, $G_{k-1}$ and $G_k$ are birecurrent. We have AM property I.

The first and second parts of AM property II are equivalent and the second part holds by (extII) for an extension and (swII) for a switch. For AM property III note that there is only a single red vertex (labeled $d^u_k$) in $G_k$ and is only a single red vertex (labeled $d^u_{k-1}$) in $G_{k-1}$ because of the requirement in (gtII) that $G_k$ and $G_{k-1}$ are $(r;(\frac{3}{2}-r))$ ltt structures (see the standard notation for why this is notationally consistent with the AM properties). What is left of AM property III is that the edge $[t^R_k]= [d^u_k, \overline{d^a_k}]$ in $G_k$ and the edge $[t^R_{k-1}]= [d^u_{k-1}, \overline{d^a_{k-1}}]$ in $G_{k-1}$ are both red. This follows from (gtI) combined with (extII) for an extension and (swII) for a switch.

(gtIII) implies AM property IV. For AM property V, note: AM property III implies $e^R_k$ is a red edge containing the red vertex $d^u_k$. (ltt(*)4) implies the uniqueness of both the red edge and direction.

Since AM property VI follows from (gtI), combined with (extII) for an extension and (swII) for a switch, and AM property VII follows from (gtIII), we have proved the converse.  \qedhere
\end{proof}

\vskip8pt

\begin{df}
In light of Proposition \ref{P:ExtensionsSwitches}, an \emph{admissible map} will mean a triple for a $\mathcal{G} \in \mathcal{PI}_{(r;(\frac{3}{2}-r))}$ that is an admissible switch or admissible extension or (equivalently) satisfies AM properties I-VII.
\end{df}

\section{Ideal decomposition ($\mathcal{ID}$) diagrams}{\label{Ch:AMDiagrams}}

\emph{\textbf{Throughout this section $\mathcal{G} \in \mathcal{PI}_{(r;(\frac{3}{2}-r))}$.}} We define the ``ideal decomposition ($\mathcal{ID}$) diagram'' for $\mathcal{G}$, as well as prove that representatives with $(r;(\frac{3}{2}-r))$ potential are realized as loops in these diagrams. We use $\mathcal{ID}$ diagrams to prove Theorem \ref{T:MainTheorem}B and to construct examples in \cite{p12d}.

\begin{df} A \emph{preliminary ideal decomposition diagram for $\mathcal{G}$} is the directed graph where
~\\
\vspace{-5.8mm}
\begin{enumerate}[itemsep=-1.5pt,parsep=3pt,topsep=3pt]
\item the nodes correspond to equivalence classes of admissible indexed pair-labeled $(r;(\frac{3}{2}-r))$ ltt structures for $\mathcal{G}$ and \\[-5.5mm]
\item for each equivalence class of an admissible generator triple ($g_i$, $G_{i-1}$, $G_i$) for $\mathcal{G}$, there exists a directed edge $E(g_i, G_{i-1}, G_i)$ from the node [$G_{i-1}$] to the node [$G_i$].
\end{enumerate}
The disjoint union of the maximal strongly connected subgraphs of the preliminary ideal decomposition diagram for $\mathcal{G}$ will be called the \emph{ideal decomposition ($\mathcal{ID}$) diagram for $\mathcal{G}$} (or \emph{$\mathcal{ID}(\mathcal{G})$}).
\end{df}

\begin{rk} \cite{p12a} gives a procedure for constructing $\mathcal{ID}$ diagrams (there called ``AM Diagrams'').
\end{rk}

\indent We say an ideal decomposition $\Gamma_0 \xrightarrow{g_1} \Gamma_1 \xrightarrow{g_2} \cdots \xrightarrow{g_{k-1}}\Gamma_{k-1} \xrightarrow{g_k} \Gamma_k$ of a tt $g$ with indexed $(r;(\frac{3}{2}-r))$ ltt structures $G_0 \to G_1 \to \cdots \to G_{k-1} \to G_k$ for $\mathcal{G}$ is \emph{realized} by $E(g_1, G_{0}, G_1) * \dots * E(g_k, G_{k-1}, G_k)$ in $\mathcal{ID}(\mathcal{G})$ if the oriented path $E(g_1, G_{0}, G_1) * \dots * E(g_k, G_{k-1}, G_k)$ in $\mathcal{ID}(\mathcal{G})$ from [$G_0$] to [$G_k$], traversing the $E(g_i, G_{i-1}, G_i)$ in order of increasing $i$ (from $E(g_1, G_{0}, G_1)$ to $E(g_k, G_{k-1}, G_k)$), exists.

\begin{prop}{\label{P:ReferenceLoop}} If $g=g_{k} \circ \cdots \circ g_1$, with ltt structures $G_0 \to G_1 \to \cdots \to G_{k-1} \to G_k$, is an ideally decomposed representative of $\phi \in Out(F_r)$, with $(r;(\frac{3}{2}-r))$ potential, such that $\mathcal{IW}(\phi)=\mathcal{G}$, then $E(g_1, G_{0}, G_1) * \dots * E(g_k, G_{k-1}, G_k)$ exists in $\mathcal{ID}(\mathcal{G})$ and forms an oriented loop. \end{prop}

\begin{proof} This follows from Proposition \ref{P:ExtensionsSwitches} and Proposition \ref{P:am}.  \qedhere
\end{proof}

\vskip5pt

\begin{cor}{\label{C:ReferenceLoop}} \textbf{(of Proposition \ref{P:ReferenceLoop})} If no loop in $\mathcal{ID}(\mathcal{G})$ gives a potentially-$(r;(\frac{3}{2}-r))$ representative of a $\phi \in Out(F_r)$ with $\mathcal{IW}(\phi) = \mathcal{G}$, such a $\phi$ does not exist. In particular, any of the following $\mathcal{ID}(\mathcal{G})$ properties would prove such a representative does not exist:
~\\
\vspace{-5.6mm}
\begin{enumerate}[itemsep=-1.5pt,parsep=3pt,topsep=3pt]
\item For at least one edge pair $\{d_i, \overline{d_i}\}$, where $e_i \in \mathcal{E}(\Gamma)$, no red vertex in $\mathcal{ID}(\mathcal{G})$ is labeled by $d_i^{\pm 1}$.
\item The representative corresponding to each loop in $\mathcal{ID}(\mathcal{G})$ has a pNp.
\end{enumerate}
\end{cor}

\noindent As a result of Corollary \ref{C:ReferenceLoop}(1) we define:

\begin{df}
\textbf{Irreducibility Potential Test:} Check whether, in each connected component of $\mathcal{ID}(\mathcal{G})$, for each edge vertex pair $\{d_i, \overline{d_i}\}$, there is a node $N$ in the component such that either $d_i$ or $\overline{d_i}$ labels the red vertex in the structure $N$. If it holds for no component, $\mathcal{G}$ is unachieved.
\end{df}

\begin{rk}
Let $\{x_1, \overline{x_1}, \dots, x_{2r}, \overline{x_{2r}}\}$ be a rank-r edge pair labeling set. We call a permutation of the indices $1 \leq i \leq 2r$ combined with a permutation of the elements of each pair $\{x_i, \overline{x_i}\}$ an \emph{Edge Pair (EP) Permutation}. Edge-indexed graphs will be considered \emph{Edge Pair Permutation (EPP) isomorphic} if there is an EP permutation making the labelings identical (this still holds even if only a subset of $\{x_1, \overline{x_1}, \dots, x_{2r}, \overline{x_{2r}}\}$ is used to label the vertices, as with a graph in $\mathcal{PI}_{(r;(\frac{3}{2}-r))}$).

When checking for irreducibility, it is only necessary to look at one EPP isomorphism class of each component (where two components are in the same class if one can be obtained from the other by applying the same EPP isomorphism to each triple in the component).
\end{rk}

\section{Several unachieved ideal Whitehead graphs}{\label{Ch:UnachievableGraphs}}

\begin{thm}{\label{T:MainTheorem}} For each $r \geq 3$, let $\mathcal{G}_r$ be the graph consisting of $2r-2$ edges adjoined at a single vertex.
~\\
\vspace{-\baselineskip}
\begin{description}
\item [A.] For no fully irreducible $\phi \in Out(F_r)$ is $\mathcal{IW}(\phi) \cong \mathcal{G}_r$. \\[-5.5mm]
\item [B.] The following connected graphs are not the ideal Whitehead graph $\mathcal{IW}(\phi)$ for any fully irreducible $\phi \in Out(F_3)$: \\[-6.5mm]

\centering
\includegraphics[width=2.6in]{UnachievableGraphs.eps}
\end{description} \end{thm}

\begin{proof} We first prove (A). By Proposition \ref{P:BC}, it suffices to show that no admissible $(r;(\frac{3}{2}-r))$ ltt structure for $\mathcal{G}$ is birecurrent. Up to EPP-isomorphism, there are two such ltt structures to consider, neither birecurrent):
~\\
\vspace{-9.25mm}
\begin{figure}[H]
\centering
 \includegraphics[width=2in]{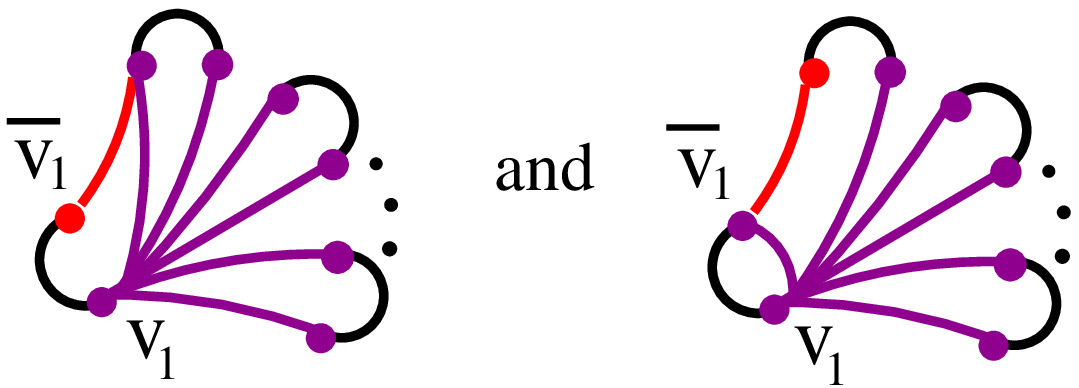}
\label{fig:NotBirecurrent} \\[-4.5mm]
\end{figure}

\noindent These are the only structures worth considering as follows: Call the valence-($2r-2$) vertex $v_1$. Either (1) some valence-1 vertex is labeled by $\overline{v_1}$ or (2) the set of valence-$1$ vertices $\{x_1, \overline{x_1}, \dots, x_{r-1}, \overline{x_{r-1}}\}$ consists of $r-1$ edge-pairs. Suppose (2) holds. The red edge cannot be attached in such a way that it is labeled with an edge-pair or is a loop and attaching it to any other vertex yields an EPP-isomorphic ltt structure to that on the left. Suppose (1) holds. Let $x_i$ label the red vertex. The valence-$1$ vertex labels will be $\{\overline{v_1}, x_2, \overline{x_2}, \dots, x_{i-1}, \overline{x_{i-1}}, \overline{x_i}, x_{i+1}, \overline{x_{i+1}}, \overline{x_i} \dots, x_{r}, \overline{x_{r}}\}$. The red edge cannot be attached at $\overline{x_i}$. So either it will be attached at $v_1$, $\overline{v_1}$, or some $x_j$ with $x_j \neq x_i^{\pm 1}$. Unless it is attached at $\overline{v_1}$, $\overline{v_1}$ is a valence-$1$ vertex of [$v_1, \overline{v_1}$] in the local Whitehead graph, making $[v_1, \overline{v_1}]$ an edge only traversable once by a smooth line. If the red edge is attached at $\overline{v_1}$, we have the structure on the right.

We prove (B). The left graph is covered by A. The following is a representative of the EPP isomorphism class of the only significant component of $\mathcal{ID}(\mathcal{G})$ where $\mathcal{G}$ is the right-most structure:
~\\
\vspace{-6.5mm}
\begin{figure}[H]
\centering
\includegraphics[width=3.7in]{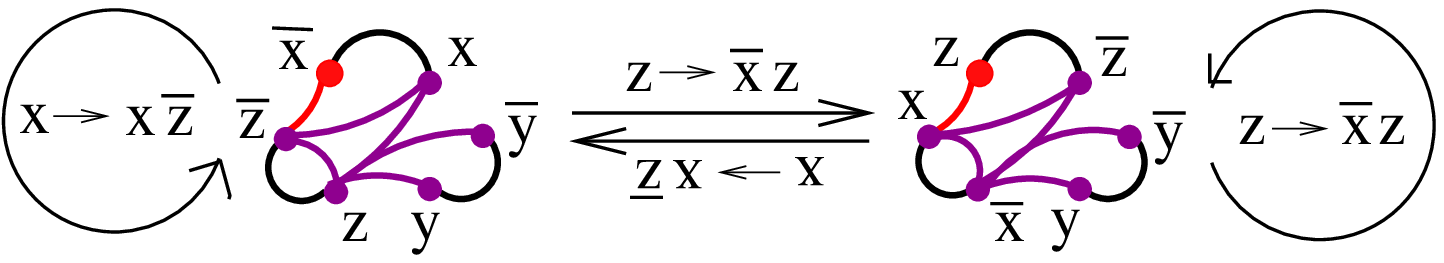}
\label{fig:IllustrativeAMDiagram} \\[-5mm]
\end{figure}

\smallskip

\indent Since $\mathcal{ID}(\mathcal{G})$ contains only red vertices labeled $z$ and $\bar{x}$ (leaving out $\{y, \overline{y}\}$), unless some other component contains all 3 edge vertex pairs ($\{x, \overline{x}\}$, $\{y, \overline{y}\}$, and $\{z, \overline{z}\}$), the middle graph would be unachieved. Since no other component does contain all 3 edge vertex pairs as vertex labels (all components are EPP-isomorphic), the middle graph is indeed unachieved.

Again, for the right-hand, the $\mathcal{ID}$ Diagram lacks irreducibility potential. A component of the $\mathcal{ID}$ diagram is given below (all components are EPP-isomorphic). The only edge pairs labeling red vertices of this component are $\{x, \overline{x} \}$ and $\{z, \overline{z}\}$:
~\\
\vspace{-6mm}
\begin{figure}[H]
\centering
\includegraphics[width=4in]{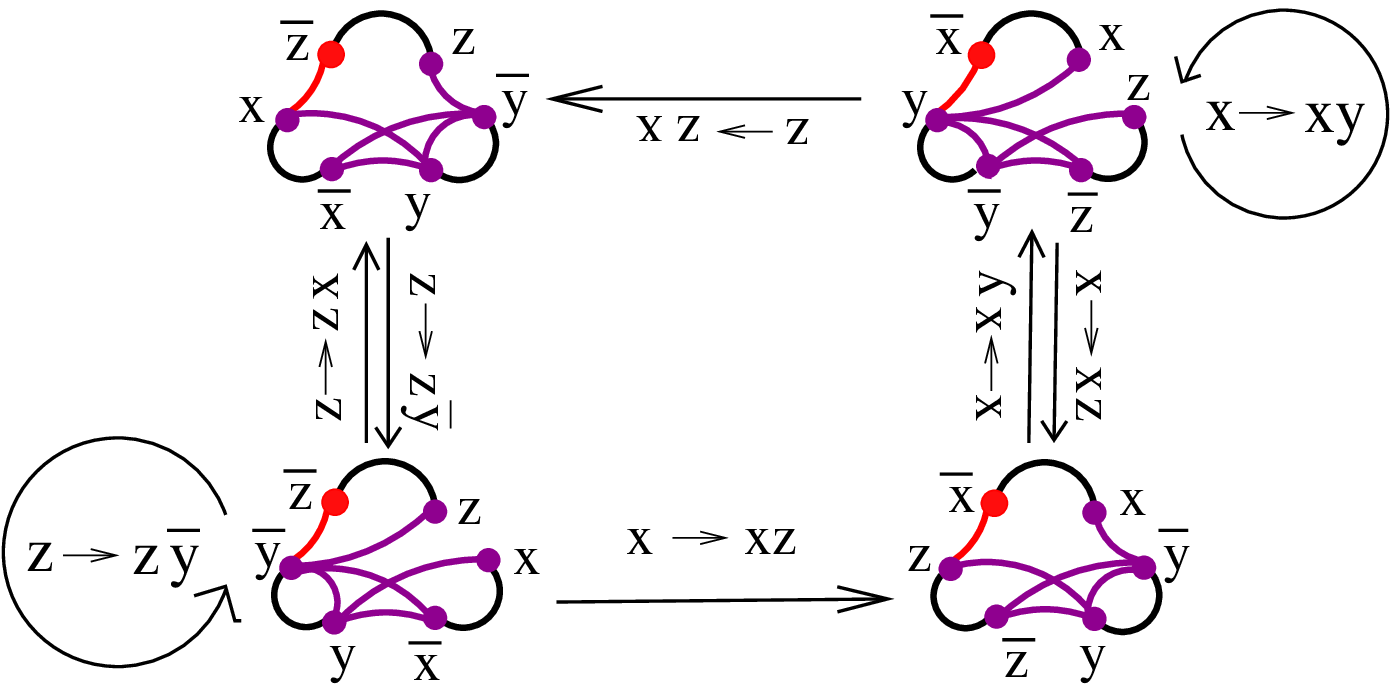}
\label{fig:IllustrativeAMDExample}
\end{figure}
\qedhere
\end{proof}

\newpage

\bibliographystyle{amsalpha}
\bibliography{BirecurrencyCondition}

\providecommand{\bysame}{\leavevmode\hbox to3em{\hrulefill}\thinspace}
\providecommand{\MR}{\relax\ifhmode\unskip\space\fi MR }
\providecommand{\MRhref}[2]{%
  \href{http://www.ams.org/mathscinet-getitem?mr=#1}{#2}
}
\providecommand{\href}[2]{#2}
\begin{thebibliography}{GJLL98}

\bibitem[BF94]{bf94}
M.~Bestvina and M.~Feighn, \emph{Outer limits}, preprint (1994), 1--19.

\bibitem[BFH97]{bfh97}
M.~Bestvina, M.~Feighn, and M.~Handel, \emph{Laminations, trees, and
  irreducible automorphisms of free groups}, Geometric and Functional Analysis
  \textbf{7} (1997), no.~2, 215--244.

\bibitem[BFH00]{bfh00}
\bysame, \emph{The tits alternative for out (fn) i: Dynamics of
  exponentially-growing automorphisms}, Annals of Mathematics-Second Series
  \textbf{151} (2000), no.~2, 517--624.

\bibitem[BH92]{bh92}
M.~Bestvina and M.~Handel, \emph{Train tracks and automorphisms of free
  groups}, The Annals of Mathematics \textbf{135} (1992), no.~1, 1--51.

\bibitem[CV86]{cv86}
M.~Culler and K.~Vogtmann, \emph{Moduli of graphs and automorphisms of free
  groups}, Inventiones mathematicae \textbf{84} (1986), no.~1, 91--119.

\bibitem[GJLL98]{gjll}
D.~Gaboriau, A.~Jaeger, G.~Levitt, and M.~Lustig, \emph{An index for counting
  fixed points of automorphisms of free groups}, Duke mathematical journal
  \textbf{93} (1998), no.~3, 425--452.

\bibitem[HM07]{hm07}
M.~Handel and L.~Mosher, \emph{Parageometric outer automorphisms of free
  groups}, Transactions of the American Mathematical Society \textbf{359}
  (2007), no.~7, 3153--3184.

\bibitem[HM11]{hm11}
\bysame, \emph{Axes in outer space}, no. 1004, Amer Mathematical Society, 2011.

\bibitem[JL09]{jl09}
A.~J{\"a}ger and M.~Lustig, \emph{Free group automorphisms with many fixed
  points at infinity}, arXiv preprint arXiv:0904.1533 (2009).

\bibitem[Mah11]{m11}
J.~Maher, \emph{Random walks on the mapping class group}, Duke Mathematical
  Journal \textbf{156} (2011), no.~3, 429--468.

\bibitem[MS93]{ms93}
H.~Masur and J.~Smillie, \emph{Quadratic differentials with prescribed
  singularities and pseudo-anosov diffeomorphisms}, Commentarii Mathematici
  Helvetici \textbf{68} (1993), no.~1, 289--307.

\bibitem[NH86]{n86}
J.~Nielsen and V.L. Hansen, \emph{Jakob nielsen, collected mathematical papers:
  1913-1932}, vol.~1, Birkhauser, 1986.

\bibitem[Pfa12a]{p12a}
C.~Pfaff, \emph{Constructing and classifying fully irreducible outer
  automorphisms of free groups}, Ph.D. thesis, Rutgers University, 2012.

\bibitem[Pfa12b]{p12c}
\bysame, \emph{Ideal whitehead graphs in $out(f_r)$ ii: Complete graphs in
  every rank}, Preprint (In Progress),
  {\url{http://www.math.rutgers.edu/~cpfaff/CompleteGraphs.pdf}}, 2012.

\bibitem[Pfa12c]{p12d}
\bysame, \emph{Ideal whitehead graphs in $out(f_r)$ iii: Achieved graphs in
  rank $3$}, Preprint,
  {\url{http://www.math.rutgers.edu/~cpfaff/MainTheorem.pdf}}, 2012.

\bibitem[Sko89]{s89}
R.~Skora, \emph{Deformations of length functions in groups, preprint}, Columbia
  University (1989).

\bibitem[Sta83]{s83}
J.R. Stallings, \emph{Topology of finite graphs}, Inventiones Mathematicae
  \textbf{71} (1983), no.~3, 551--565.

\end{thebibliography}

\end{document}